\documentclass[twoside]{amsart}

\usepackage[utf8]{inputenc}
\usepackage[T1]{fontenc}
\usepackage[english]{babel}

\usepackage{mathpazo}
\usepackage{mathptmx}
\usepackage{amsmath}
\usepackage{amssymb}
\usepackage{mathrsfs}
\usepackage{amsthm}
\usepackage{amsfonts}
\usepackage{fancyhdr}
\usepackage[all]{xy}

\usepackage[left=3cm, right=3cm]{geometry}

\date{}

\usepackage{amsfonts, amssymb, amsmath, amsthm, multicol,bbm}
\usepackage[colorlinks=true, pdfstartview=FitV, linkcolor=blue,
            citecolor=blue, urlcolor=blue]{hyperref}

\numberwithin{equation}{section}

\newtheorem*{Thm*}{Theorem}

\usepackage{cjhebrew}

\newcommand{\E}{\mathbb{E}}

\newcommand{\al}{\alpha}

\newcommand{\T}{T}

\usepackage[numbers,sort&compress]{natbib}

{\theoremstyle {definition} \newtheorem {defi} {Definition} [section] }
{\theoremstyle {plain}  \newtheorem {thm} [defi] {Theorem}}
{\theoremstyle {plain}  \newtheorem {cor} [defi]{Corollary}}
{\theoremstyle {plain} \newtheorem {prop} [defi]{Proposition}}
{\theoremstyle {plain} \newtheorem {nem}[defi] {Lemma}}
{\theoremstyle {plain} }
{\theoremstyle {plain} \newtheorem {rmq}[defi] {Remark}}

{\theoremstyle {plain}  }
{\theoremstyle {plain}  }
{\theoremstyle {plain}  }
{\theoremstyle {plain} }
{\theoremstyle {plain} }
{\theoremstyle {plain} }
{\theoremstyle {plain} }

\def\E{{\Bbb{E}}}
\def\T{{\Bbb{T}}}
\def\P{{\Bbb{P}}}
\def\R{{\Bbb{R}}}
\def\C{{\Bbb{C}}}

\def\N{{\Bbb{N}}}

\def\i{{\textbf{i}}}

\def\dt{{\partial_t}}

\def\pp{{\text{ < }}}

\def\lleq{{\ \lesssim\ \ }}

\usepackage{tikz}
\usetikzlibrary{matrix}

\begin{document}

\author[M. Sy]{Mouhamadou Sy}

\address{
\newline 
\indent Department of Mathematics, Imperial College London \indent 
\newline \indent  London SW7 2AZ, United Kingdom\indent }
\email{m.sy@imperial.ac.uk}

\title[Energy supercritical NLS on $\T^3$]{Almost sure global well-posedness for the energy supercritical Schrödinger equations}

\begin{abstract}
   We consider the  Schr\"odinger equations with arbitrary (large) power non-linearity on the three-dimensional torus. We construct non-trivial probability measures supported on Sobolev spaces and show that the equations are globally well-posed on the supports of these measures, respectively. Moreover, these measures are invariant under the flows that are constructed. Therefore, the constructed solutions are recurrent in time.\\ Also, we show \textit{slow growth} control  on the time evolution of the solutions. A generalization to any dimension is given. Our proof relies on a new approach combining the fluctuation-dissipation method and some features of the Gibbs measures theory for Hamiltonian PDEs. The strategy of the paper applies to other contexts.\\
   
\noindent
\textbf{R\'esum\'e}: Nous consid\'erons des \'equations de Schr\"odinger avec des puissance arbitrairement grandes de la nonlin\'earit\'e sur le tore tri-dimensionel. Nous construisons des mesures de probabilit\'e non triviales support\'ees sur des espaces de Sobolev et montrons que les \'equations sont globalement bien pos\'ees sur les supports de ces mesures, respectivement.  De plus, ces mesures sont invariantes sous les flots qui sont construits.  Par cons\'equent, les solutions construites sont r\'ecurrentes en temps.\\ Nous \'etablissons \'egalement des contr\^oles de \textit{faible croissance} sur l'\'evolution en temps des solutions.  Une g\'en\'eralisation \`a toutes dimensions est fournie.  Notre preuve se base sur une nouvelle approche combinant la m\'ethode de fluctuation-dissipation  et certaines caract\'eristiques de la th\'eorie des mesures de Gibbs pour les EDPs hamiltoniennes. La strat\'egie de cet article s'applique \`a d'autres contextes. \\


\end{abstract}

\maketitle

\nocite{*}

\section{Introduction}
\subsection{Context} We consider the following Schr\"{o}dinger equations
%
\begin{equation}
{\partial _{t}}u={\textbf{i}}(\Delta u-|u|^{p-1}u),
\label{Equ_NLS7-1}
\end{equation}
where ${\textbf{i}}$ is the complex number that satisfies
${\textbf{i}}^{2}=-1$. Here we consider that the unknown function
$u=u(t,x)$ is defined on ${\mathbb{R}}\times {\mathbb{T}}^{3}$ and takes values
in $\mathbb{C}$,
${\mathbb{T}}^{3}=\left (\frac{{\mathbb{R}}}{2\pi \mathbb{Z}}\right )^{3}$ is a
three-dimensional torus. Throughout the paper, the parameter $p$ can be
taken to be any real number bigger or equal to $3$, although we highlight
the case $p>5$ where the equation is much less understood.

The present work is devoted to proving a probabilistic global well-posedness
for \eqref{Equ_NLS7-1} and to establishing strong controls on the growth
in time of the solutions by using an invariant measure technique. Before
we present our method and results, let us recall the structure of the
equation and some known results. The equation \eqref{Equ_NLS7-1} is Hamiltonian
and is derived from the energy
\begin{align*}
H(u) &=\int _{{\mathbb{T}}^{3}}\left (\frac{1}{2}|\nabla u(t,x)|^{2}+
\frac{1}{p+1}|u(t,x)|^{p+1}\right )dx.
\end{align*}
It also preserves the following quantity (its mass)
%
\begin{align}
M(u) &=\frac{1}{2}\int _{{\mathbb{T}}^{3}} |u(t,x)|^{2}dx,
\label{MassNLS}
\end{align}
and obeys the scaling invariance
%
\begin{equation}
u_{\lambda }(t,x)=\lambda u(\lambda ^{p-1}t,\lambda ^{\frac{p-1}{2}}x).
\end{equation}
A direct computation shows that, for the homogeneous Sobolev norm of order
$s$,
\begin{align*}
\|D^{s}u_{\lambda }\|=\lambda ^{1+(p-1)\left (\frac{s}{2}-\frac{3}{4}
\right )}\|D^{s}u\|.
\end{align*}
It follows that the critical exponent of \eqref{Equ_NLS7-1} is
$s_{c}=\frac{3}{2}-\frac{2}{p-1}$. For $p>5$, the space $H^{s_{c}}$ is
smoother than the energy space $H^{1}$, we have that \eqref{Equ_NLS7-1} is energy supercritical. This results in that the global
regularity problem of \eqref{Equ_NLS7-1} is an extremely challenging open
problem. We do not know any global well-posedness result on a Sobolev space
for \eqref{Equ_NLS7-1} in dimension $3$. In the best of our knowledge,
the Cauchy problem for \eqref{Equ_NLS7-1} considered in
${\mathbb{R}}^{3}$ is solved only locally in time or globally for small data
on some Sobolev spaces \cite{cazweis}. On the torus
${\mathbb{T}}^{3}$, we can establish a local theory as well, at least for
smooth enough data, for example beyond the $H^{\frac{3}{2}}$-regularity.
But, at this level of regularity, the energy does not control the relevant
norm, and then cannot be an argument of globalization. In the present work,
we globalize the local theory of \eqref{Equ_NLS7-1} posed on
$H^{s}({\mathbb{T}}^{3})$, $s\geq 2$, for data belonging to subsets of these
spaces, where we have proved a control on the relevant norm. These subsets
are constructed using a probabilistic method and they satisfy some qualitative
and quantitative properties. For example, we prove that they contain data
of any given size.

The critical and subcritical cases of \eqref{Equ_NLS7} ($p\leq 5$) were
widely studied by many authors (see
\cite{bourfocS,bourgain99,GrllakscrtNLS,BGT04,bgtGS,bgtcoll,CKSTTcrtNLS,htt,ionpaus,ptw}
and references therein), in particular global well-posedness of the Cauchy
theory was proved, as well as long time dynamics properties. 

 For the supercritical setting ($p>5$), let us mention the work of Wang
\cite{wang} who used bifurcation analysis to construct global quasi-periodic
solutions to energy supercritical NLS and Beceanu-Deng-Soffer-Wu
\cite{beceanu} who constructed global scattering solutions for the equation
as well. A difference between the present setting and that of these two
papers is that, in one hand, the solutions we investigate are not small
compared to those in \cite{wang}; on the other hand, our physical space
is the torus compared to $\mathbb{R}^{3}$ in \cite{beceanu}. In particular
the solutions constructed in the present paper do not scatter. However,
we do not know if the small solutions contained in the constructed statistical
ensemble have a quasi-periodicity property as for those in
\cite{wang}.

\subsection{Invariant measures methods in dispersive PDEs}
\label{sec1.2}

The Gibbs measures techniques for dispersive PDEs go back to Lebowitz,
Rose and Speer \cite{LRS}, where one-dimensional nonlinear Schr\"{o}dinger
equations (1D NLS) are considered and Gibbs measures are constructed, the
question of existence of global flow matching the regularity of the support
of the measure was left open, and then so does the question of invariance.
Zhidkov \cite{zhidnls} redefined these measures via finite-dimensional
approximations and a passage to the limit, he showed the invariance in
the case of the wave equation \cite{zhidwave}. Bourgain
\cite{bourg94,bourgNLS96} constructed global flow and proved the invariance
of the measures in question for 1D (subquintic) NLS. He then showed similar
result for the 2D defocusing cubic Schr\"{o}dinger equation posed on
${\mathbb{T}}^{2}$, under a suitable renormalization. Tzvetkov
\cite{tzvNLS06,tzvNLS} considered the subquintic NLS equations (including
the focusing nonlinearity for the subcubic case) on the disc of
${\mathbb{R}}^{2}$ and constructed invariant Gibbs measures, he proved a
probabilistic global well-posedness on relevant spaces. Three-dimensional
results have followed in
\cite{bt2007,btrandom2,asds,bourbulNLS,bourbulW}. Several other contexts
were analyzed on the line of these techniques. An important feature in
this approach is that the Gibbs measures are concentrated on relatively
rough spaces. Namely, their supports are $\frac{d}{2}+$ degrees of regularity
weaker than that of the energies on which they are based. Here $d$ is the
(effective) dimension of the physical space, and is sensitive to some symmetry
assumptions. That is why these objects are often used to deal with the
global well-posedness on spaces of low regularity.

A second approach to construct invariant measures is the so called fluctuation-dissipation
method, based on `compact approximations' of the equation and a use of
stochastic tools. An inviscid limit is then considered. We go back to the
work of Kuksin \cite{kuk_eul_lim} and Kuksin and Shirikyan
\cite{KS04,KS12} for the 2D Euler equation and the cubic defocusing Schr\"{o}dinger
equation (in dimension $\leq 4$). For both of these equations, an invariant
measure on the Sobolev space $H^{2}$ is obtained. Let us also mention some
results of the author \cite{sybo,sykg} using this approach.

It is worth mentioning some `non-invariance' probabilistic methods in the
Cauchy problem of PDEs, see for instance
\cite{btrandom1,thomann2009random,colloh,MR3131480,burktzvt_probwav,pocovnicu2014,ohpocov}.

Let us also notice the works \cite{tzvQI1,tzQI2,tzQI3} that showed quasi-invariance
properties of Gaussian measures under Hamiltonian flows.

\subsection{The methodology of the paper}
\label{sec1.3}

In the present work, we introduce a new approach which is somehow hybrid
to deal with the supercriticality present in \eqref{Equ_NLS7-1}. Namely,
we combine the fluctuation-dissipation method with some features of the
approach based on Gibbs measures. Indeed, the nonlinearity is such that
usual energy methods do not provide continuity of the flow in the context
the standard fluctuation-dissipation for Hamiltonian PDEs, essentially because of a lack of time integrability. Also, the Gibbs
measure approach does not work either due to a lack of space regularity obstructing the analysis of (the large power of) the nonlinearity
in dimension 3 and higher. Our solution to overcome these two serious
issues is to combine the two approaches in a single new one. We plug fluctuation-dissipation
tools into the setting of the Gibbs measures as developed in
\cite{bourg94}. More precisely, we consider damped/driven Galerkin approximations
of the NLS \eqref{Equ_NLS7-1}, construct invariant measures that enjoy
bounds that are uniform both in the viscosity parameter and in the dimension
of the approximating equation. We pass first to the limit when the viscosity
goes to $0$ and obtain a sequence of invariant measures associated to the
(deterministic) Galerkin approximations for \eqref{Equ_NLS7-1}. We study
the infinite-dimensional limit in spirit of Bourgain \cite{bourg94}. In
order to obtain large deviation bounds that are exploitable in the Bourgain
argument, we introduce a new and carefully prepared dissipation operator
as discussed below. Once these bounds are obtained, we perform an extension
of the Bourgain framework to the measures that are not necessarily of Gibbs
type. Despite the lack of information on our measures, occasioned by the
compactness method, we were able to make the infinite-dimensional data
(living on the support of the limiting measure) inheriting the good properties
of their finite-dimensional approximations. To achieve this, a suitable
family of restriction measures (that are conditional probabilities) is
introduced and the Skorokhod representation theorem is used.

Let us discuss a first new ingredient of our proof: in a step of the approximation
argument (see Section~\ref{ASection3flucdiss}), we considered fluctuation-dissipation
on the Galerkin equation having this form:
\begin{align*}
du={\textbf{i}}[(\Delta -1) u-P_{N}(|u|^{p-1}u)]dt-\alpha [(1-\Delta )^{s-1}+e^{
\rho (\|u\|_{s-})}]udt+\sqrt{\alpha }d\eta _{N},
\end{align*}
where $\eta _{N}$ is a noise and $s- =s-\epsilon $, for some
$\epsilon >0$ close enough to $0$ (we use $s+$ in a similar way). Let us
focus on the factor $e^{\rho (\|u\|_{s-})}$ in the dissipation operator.
Remark that an application of the It\^{o} formula to the mass, given by \eqref{MassNLS}, provides a statistical control (under an eventual invariant
measure) of the quantity
\begin{align*}
{\mathbb{E}}e^{\rho (\|u\|_{s-})}\|u\|^{2} ,
\end{align*}
and therefore, to some extent, of the quantity
\begin{align*}
{\mathbb{E}}e^{\rho (\|u\|_{s-})}.
\end{align*}
Such a control is a crucial step in the use of the argument of Bourgain
\cite{bourg94} (see also \cite{tzvNLS06,tzvNLS}). The local existence time
for \eqref{Equ_NLS7} depends on the size of the data (that is in our situation
$T\sim \|u_{0}\|_{\frac{3}{2}+}^{1-p}$). Without the factor
$e^{\rho (\|u\|_{s-})}$, we should control only a quadratic power, which
seems to be not enough (see the proof of Proposition \ref{prop_control}). The `miracle' of this factor is that, beside its exponential
strength, it does not participate to usual estimation computations (such
as integration in $x$, projection, etc.) And it is of an assigned regularity.
That means its regularity is chosen and does not directly depend on the
structure of the equation; the function $\rho $ is also our choice modulo
some weak constraints. Furthermore, the stronger this factor the slower
the growth in time of the constructed solutions. This new ingredient should
be a trick that can be used in many other situations and its flexibility
can be exploited further. More generally, all this new approach might provide
a new way to construct global solutions and invariant measures and to establish
slow growth properties, specially for PDEs presenting strong supercriticality.%

The overall message of this strategy is that, by employing the performed
generalization of the Bourgain framework, one can extract individual (pointwise)
bounds from the statistical (integral) ones provided by the fluctuation-dissipation setting.
The regularity of the controlled quantities allows, in particular, to obtain
uniqueness and continuity properties (that are missing if we use only the classical
fluctuation-dissipation). The other bounds arise as a byproduct of the
method.

\subsection{Main results}

Set $v=e^{-{\textbf{i}}t}u$, we see that if $u$ solves \eqref{Equ_NLS7-1}, then $v$ solves the following equation
%
\begin{equation}
\label{Equ_NLS7}
{\partial _{t}}v={\textbf{i}}[(\Delta -1) v-|v|^{p-1}v].
\end{equation}
This formulation is more appropriate for dealing with the zero frequency.
Its energy is the following
\begin{align*}
E(v)=\int _{{\mathbb{T}}^{3}}\frac{1}{2}|v|^{2}+\frac{1}{2}|\nabla v|^{2}+
\frac{1}{p+1}|v|^{p+1}dx.
\end{align*}
Let us state the main result of the paper.
%
\begin{thm}%
\label{ThmPrincipal}
For any $s\geq 2$ and any increasing concave function
$\xi :{\mathbb{R}}_{+}\to {\mathbb{R}}_{+}$, there is a measure
$\mu =\mu _{s,\xi }$ concentrated on $H^{s}$ such that
\begin{enumerate}
\item
\label{Mainstat1}%
for $\mu $-almost any $u_{0}\in H^{s}$, there is a unique solution
$u\in C({\mathbb{R}},H^{s})$ to\eqref{Equ_NLS7} such that
$u(0)=u_{0}$;
\item the distributions of $M(u)$ and $E(u)$ via $\mu $ admit densities
with respect to the Lebesgue measure on ${\mathbb{R}}$.
\item
\label{MainBigness}%
For any $n >0$, there is a set $S_{n}$ such that $\mu (S_{n})>0$, and for
any $u_{0}\in S_{n}$, $\|u_{0}\|_{s}\geq n$.
\item The flow $\phi ^{t}$ implied by the statement \ref{Mainstat1} satisfies
the following properties:
\begin{enumerate}
\item For any $T_{0}>0$, there is $C(T_{0})>0$ such that for $\mu $-almost
any $u$ and $v$ in $H^{s}$ we have
\begin{align*}
\sup _{t\in [-T_{0},T_{0}]}\|\phi ^{t}u-\phi ^{t}v\|_{s}\leq C(T_{0})
\|u-v\|_{s}.
\end{align*}
\item The measure $\mu $ is invariant under $\phi ^{t}$.
\item For $\mu -$ almost all $u_{0}\in H^{s}$ we have the slow growth property
%
\begin{equation}
\label{Main}
\|\phi ^{t}u_{0}\|_{s-}\leq C_{\xi }(\|u_{0}\|_{s-})\xi (\ln (1+|t|))
\quad \text{for all $t\in {\mathbb{R}}$}.
\end{equation}
\end{enumerate}
\end{enumerate}
\end{thm}
Consequently, using the Poincar\'{e} recurrence theorem, we have that for
$\mu $-almost any $u_{0}\in H^{s}$, there is a sequence
$t_{k}\uparrow \infty $ such that
\begin{align*}
\lim _{t_{k}\to \infty }\|\phi ^{\pm t_{k}}u_{0}-u_{0}\|_{s}=0.
\end{align*}
That gives a long time property of the flow $\phi ^{t}$.%

Also, the point \ref{MainBigness} of the statement expresses the fact that
our result is not of small data type.

Now, in the control \eqref{Main} the function $\xi $ being concave gives
us the desired slow growth: concretely we can choose $\xi $ to be
$\log $ or $\log \circ \log $ or even `better' (as long as it is increasing
concave), and our result ensures then the existence of a measure with the
mentioned qualitative/quantitative properties which provides the chosen
control on the growth of the solutions.%
\\ This may be put in contrast with the Gibbs measure techniques, where the
Fernique theorem provides an estimate of the quantity
${\mathbb{E}}e^{c\|u\|_{r}^{2}}$, giving rise to a control of type
$\sqrt{\ln (1+|t|)}$.%

 To end our discussion let us state the generalization of the results to
all dimensions, specially to dimensions $d> 3$\textup{:}
%
\begin{rmq}
The results of this paper remain true if we consider the equation \eqref{Equ_NLS7} to be posed on a $d\-$dimensional torus
${\mathbb{T}}^{d}$, $d\geq 3$. The computations are the same modulo some
minor adaptations in the conditions of some statements. However, we need
to assume that $s>\frac{d}{2}$ so that Proposition \ref{propLWP} survives.
\end{rmq}

\subsection{Organization of the paper}

In  section~\ref{ASection2UniformLWP} we present a local well-posedness
result of \eqref{Equ_NLS7} and its Galerkin approximations on smooth spaces.
We emphasis on the fact that the time of existence can be taken independently
of the dimension. We show a convergence result.%
\\ In section~\ref{ASection3flucdiss}, we study the fluctuation-dissipation
equations based on the Galerkin approximations of \eqref{Equ_NLS7}, we
establish stochastic global well-posedness, existence of stationary measures
and we derive uniform estimates for them. Then, the section~\ref{ASection4Invisc} is devoted to the study of inviscid limits, it is
shown that the inviscid measures are invariant under the flows of the approximating
problems for \eqref{Equ_NLS7}, uniform in $N$ bounds are proved. In section~\ref{ASection5StatensemGWP}, we construct the infinite-dimensional statistical
ensemble, derive bounds for the approximating dynamics, and use them to
construct global flows for \eqref{Equ_NLS7} on $H^{r}$ for
$r\leq s-$ and for data living on the statistical ensemble. Section~\ref{ASection6InvarMeas} is concerned with the invariance of the infinite-dimensional
limiting measure. In section~\ref{ASection7GWPH2Size}, we use an argument
based on the propagation of regularity principle to state the almost sure
global wellposedness with respect to the $H^{s}$-regularity. We also deal
with the size of the data by constructing a cumulative measure. And finally,
we derive qualitative properties for the constructed measure in Section~\ref{ASection8Qualprop}.
\subsection*{General notations}
Consider the sequence $\left((2\pi)^{\frac{-3}{2}}e^{\i k\cdot x}\right)_{k\in\N^3}$ whose elements are normalized eigenfunctions of the Laplace operator $-\Delta$ on $\T^3=\left(\frac{\R}{2\pi\Bbb Z}\right)^3$. The associated eigenvalues are $|k|^2=k_1^2+k_2^2+k_3^2.$  We shall arrange the eigenfunctions in the increasing order of eigenvalues. Namely, denoting the latters as $0=\lambda_0<\lambda_1\leq \lambda_2\leq\cdots\lambda_m\leq\cdots$, we obtain the corresponding sequence of eigenfunctions  $(e_m)_{m\in\N}.$ The Weyl asymptotic states that $\lambda_m\sim  m^{\frac{2}{3}}.$ Let us denote by $e_{-m}$ the eigenfunction $\i e_m,$ we have that the sequence $(e_m)_{m\in\Bbb Z}$ forms a basis of $L^2=L^2(\T^3,\C).$\\
Therefore for $u\in L^2$, we have the representation
\begin{equation*}
u(x)=\sum_{m\in\Bbb Z}u_me_m(x).
\end{equation*}
We have the Parseval identity
\begin{align*}
\|u\|_{L^2}^2=:\|u\|^2=\int_{\T^3}|u(x)|^2dx=\sum_{m\in\Bbb Z}|u_m|^2.
\end{align*}

Let $s>0$. The Sobolev space $H^s:=H^s(\T^3;\Bbb C)$ is defined by the norm
\begin{align*}
\|u\|_s:=\sqrt{\|(1-\Delta)^{\frac{s}{2}}u\|^2}=\sqrt{\sum_{m\in\Bbb Z}(1+\lambda_m)^s|u_m|^2}.
\end{align*}
Since $\lambda_m$ are all non-negative integers, we have that $(1+\lambda_m)^s\leq(1+\lambda_m)^r$ for any $m$, if $s\leq r.$ We see then the embedding inequality
\begin{equation}\label{IneqEmbedding}
\|u\|_s\leq \|u\|_r\quad if\ \ s\leq r.
\end{equation}
 Let us define a real inner product on $L^2$ by
\begin{equation}
(u,v)=\Re \int_{\T^3} u(x)\bar{v}(x)dx,
\end{equation}
where $\Re z$ stands for the real part of the complex number $z$. Hence, we have the property
\begin{align}
(u,\i u)=0.
\end{align}

We denote by $E_N$ the subspace of $L^2$ generated by the finite family $\{e_m,\ |m|\in [0,N]\}$, the operator $P_N$ is the projector onto $E_N.$

For a functional $F: L^2\to \C,$ we denote by $F'(u;v)$ and $F''(u;v,w)$ its first derivative at $u$, evaluated at $v$, and its second derivative at $u$ evaluated at $(v,w)\in L^2\times L^2,$ respectively.\\

On the space $E_N$, we define a Brownian motion by
\begin{align}
\zeta_N(t,x)=\sum_{|m|\leq N}a_m\beta_m(t)e_{m}(x)\label{formula_BM}
\end{align}
where $(a_m)$ is a family of complex numbers, and $(\beta_m)$ is a sequence of independent standard real Brownian motions with respect to a filtration $(\mathcal{F}_t)$ and defined on a probability space $(\Omega,\mathcal{F}, \P)$.\\ The noise $\eta$ is defined as
\begin{align*}
\eta_N(t,x)=\frac{d}{dt}\zeta_N(t,x).
\end{align*}
Set the numbers
\begin{align*}
A_{s,N}=\sum_{|m|\leq N}|a_m|^2\lambda_m^s.
\end{align*}

For a Banach space $H$, we denote by $C_b(H)$ the space of bounded continuous functions on $H$ with range in $\R,$ and $\mathfrak{p}(H)$ the set of all the probability measures on $H$. \\
For a Banach space $X$ and an interval $I\subset\R$, we denote by $C(I,X)=C_tX$ the space of continuous functions $f:I\to X.$ The corresponding norm is $\|f\|_{C_tX}=\sup_{t\in I}\|f(t)\|_X.$\\
For $q\in [1,\infty),$ we also denote by $L^q(I,X)=L^q_tX,$ the Lebesgue's spaces given by the norm 
\begin{align*}
\|f\|_{L^q_tX}=\left(\int_I\|f(t)\|_X^qdt\right)^{\frac{1}{q}}.
\end{align*}
The inequality $A\lleq B$ between two positive quantities $A$ and $B$ means $A\leq C B$ for some $C>0.$\\
For a measure $\mu,$ we denote by $\text{Supp}(\mu)$ the support of $\mu.$

\section{Uniform local well-posedness and (deterministic) convergence}
\label{ASection2UniformLWP}
We have the following expansion of an element $u$ in $E_N=P_NL^2$:
\begin{align*}
u=\sum_{|m|\leq N}u_me_m(x).
\end{align*}
Notice that $E_\infty$ refers to $L^2,$ and the projector $P_\infty$ to the identity operator.

In the sequel, the notation $B_R(X)$ refers to the closed ball with center $0$ and radius $R>0$ of the Banach space $X$. 

Let us consider the problem
\begin{align}
\dt u &=\i[(\Delta-1) P_Nu-P_N(|P_Nu|^{p-1}P_Nu)], \label{equations}\\
u(t_0) &=P_Nu_0.\label{datum}
\end{align} 
\subsection{Uniform LWP}
\begin{prop}\label{propLWP}
Let $s>\frac{3}{2}$. For any $R>0,$ there is a constant $T:=T(R,s)$ such that for every $N\in \N^*\cup\{\infty\}$, any $u_0\in B_R(H^s),$ there is a unique $u_N\in X^s_T:= C((-T,T),H^s)$ satisfying \eqref{equations} and \eqref{datum}.
Moreover, we have
\begin{align}
\|u\|_{X^s_T}:=\sup_{t\in (-T,T)}\|u(t)\|_s\leq 2\|P_Nu_0\|_s.\label{control_local}
\end{align}
\end{prop}
\begin{proof}
Fix $u_0$ in $B_R(H^s)$ and set the map
\begin{align*}
F(u)=S(t)P_Nu_0-\i\int_{0}^tS(t-\tau)P_N(|P_Nu|^{p-1}P_Nu)d\tau\quad C_tH^s\to C_tH^s,
\end{align*}
where $S(t)$ stands for the group $e^{\i t(\Delta-1)}.$\\
We see that an eventual fixed point of $F$ must belong to $E_N$ and be a solution to $\eqref{equations}, \eqref{datum}.$\\
Using the algebra structure of $X^s_T$, we have
\begin{align*}
\|F(u)\|_{X^s_T}\leq \|P_Nu_0\|_s+\int_0^T\|u\|_s^pd\tau\leq \|u_0\|_s+T\|u\|_{X^s_T}^p.
\end{align*}
For $T\leq \frac{1}{2^7R^{p-1}c}$ for some constant $c\geq1$, we have for all $u\in B_{2R}(X_{T}^s)$
\begin{align*}
\|F(u)\|_{X_{T}^s}\leq 2R,
\end{align*}
hence $F(u)\in B_{2R}(X_{T}^s).$ Now let $u_1$ and $u_2$ be two element of $B_{2R}(X_{T}^s)$, we have that
\begin{align*}
\|F(u_1)-F(u_2)\|_{X^s_T}\leq CT_R(\|u_1\|^{p-1}_{X^s_T}+\|u_2\|_{X^s_T}^{p-1})\|u_1-u_2\|_{X^s_T}.
\end{align*}
Taking $c=C\vee 1$ in the choice of $T$, we obtain
\begin{align*}
\|F(u_1)-F(u_2)\|_{X^s_T}\leq\frac{1}{2}\|u_1-u_2\|_{X^s_T}.
\end{align*} 
Therefore $F$ is a contraction on $B_{2R}(X_{T}^s)$ and we obtain the claimed existence and uniqueness.\\
Now, to see the last claim, let us observe that the constructed solution stay in $B_{2R}(X)$ for $|t|<T.$ Therefore, using the Duhamel formula, we have			
\begin{align*}
\|u\|_{X^s_T}\leq \|P_N u_0\|_s+\frac{1}{2}\|u\|_{X^s_T},\quad |t|<T.
\end{align*}		
This is \eqref{control_local}.
\end{proof}
\begin{rmq}
An important property of the local time existence $T$ in the Proposition \ref{propLWP} above is that it does not depend on $N.$
\end{rmq}
\subsection{Local uniform convergence}
\begin{nem}\label{LemUnifConv}
Let $s>\frac{3}{2}$, $R>0$ and $B_R:=B_R(H^s)$. Let $T:=T(s,R)$ be the associated (uniform) existence time for the problem \eqref{equations},\eqref{datum}, we have that for every $r<s$,
\begin{align*}
\sup_{u_0\in B_R}\|\phi^t(u_0)-\phi_N^t(P_Nu_0)\|_{X^r_{T}}\to 0,\ \ as\ N\to\infty.
\end{align*}
\end{nem}
\begin{proof}
Let us write the Duhamel formulas of $\phi^t(u_0)$ and $\phi_N^t(u_0)$:
\begin{align*}
\phi^t(u_0) &=S(t)u_0 -\i\int_0^tS(t-\tau)|\phi^{\tau}(u_0)|^{p-1}\phi^{\tau}(u_0)d\tau, \\
\phi_N^t(P_Nu_0) &=S(t)P_Nu_0-\i\int_0^tS(t-\tau)P_N(|\phi_N^{\tau}(P_Nu_0)|^{p-1}\phi_N^{\tau}(P_Nu_0)).
\end{align*}
Taking the difference between the two equations above and using the decomposition $f=P_Nf+f-P_Nf$, we obtain for any $t\in [0,T_R)$ that
\begin{align*}
\phi^t(u_0)-\phi_N^t(P_Nu_0) =S(t)(u_0-P_Nu_0) &-\i\int_0^tS(t-\tau)\left(P_N(|\phi^{\tau}(u_0)|^{p-1}\phi^{\tau}(u_0)-|\phi_N^{\tau}(P_Nu_0)|^{p-1}\phi_N^{\tau}(P_Nu_0))\right)d\tau\\
&-\i\int_0^tS(t-\tau)\left(|\phi^{\tau}(u_0)|^{p-1}\phi^{\tau}(u_0)-P_N(|\phi^{\tau}(u_0)|^{p-1}\phi^{\tau}(u_0))\right)d\tau.
\end{align*}
Now we use the fact that $\|P_Nf\|_r\leq\|f\|_r$ and $\|S(t)\|_{H^r\to H^r}\leq 1,$ to obtain
\begin{align*}
\|\phi^t(u_0)-\phi_N^t(P_Nu_0)\|_s\leq \|(1-P_N)u_0\|_r &+\int_0^t\||\phi^{\tau}(u_0)|^{p-1}\phi^{\tau}(u_0)-|\phi_N^{\tau}(P_Nu_0)|^{p-1}\phi_N^{\tau}(P_Nu_0)\|_r\\
&+\int_0^t\|(1-P_N)(|\phi^{\tau}(u_0)|^{p-1}\phi^{\tau}(u_0))\|_rd\tau.
\end{align*}
Now using the algebra structure of $H^s$ and the fact that, on $[0,T_R)$ we have  $\|\phi^\tau u_0\|_{L^\infty},\|\phi_N^\tau u_0\|_{L^\infty}\leq const(R),$ we obtain
\begin{align*}
\||\phi^{\tau}(u_0)|^6\phi^{\tau}(u_0)-|\phi_N^{\tau}(P_Nu_0)|^{p-1}\phi_N^{\tau}(P_Nu_0)\|_r &\leq C(\|\phi^{\tau}(u_0)\|^{p-1}_{L^\infty}+\|\phi_N^{\tau}(P_Nu_0)\|^{p-1}_{L^\infty})\|\phi^{\tau}(u_0)-\phi_N^{\tau}(P_Nu_0)\|_r\\
&\leq C(s,R)\|\phi^{\tau}(u_0)-\phi_N^{\tau}(P_Nu_0)\|_r.
\end{align*}
Remark that for $r<s$ and $f\in H^s,$ we have
\begin{align*}
\|(1-P_N)f\|_r\leq (1+\lambda_N)^{\frac{r-s}{2}}\|(1-P_N)f\|_s\leq (1+\lambda_N)^{\frac{r-s}{2}}\|f\|_s.
\end{align*}
We use the Gronwall lemma to get
\begin{align*}
\|\phi^t(u_0)-\phi_N^t(P_Nu_0)\|_r\leq(1+\lambda_N)^{\frac{r-s}{2}}e^{tC_1(s,R)}\left(\|u_0\|_s+C_2(s,R)\right).
\end{align*}
Whence follows
\begin{align*}
\sup_{u_0\in B_R}\|\phi^t(u_0)-\phi_N^t(P_Nu_0)\|_{X^r_{T}}\leq(1+\lambda_N)^{\frac{r-s}{2}}C_3(s,R).
\end{align*}
We finish the proof by recalling that $r<s$ and letting $N$ go to $\infty$.
\end{proof}

\paragraph*{A sufficient condition of globalization.}\label{subsectglobalizationcondition}
Now let us remark the following a priori bound
\begin{align*}
\|\phi^t u_0\|_{s}\leq e^{C\int_{0}^t\|\phi^\tau u_0\|_{L^\infty}^{p-1}d\tau}\|u_0\|_s.
\end{align*}
Then, if for some initial datum $u_0\in H^s$ we have that 
\begin{align}
\int_{0}^t\|\phi^\tau u_0\|_{L^\infty}^{p-1}d\tau<\infty\ \ \ for\ any\ t>0,
\end{align}
 then the solution $\phi^tu_0$ is global in time on $H^s$. 
\section{Fluctuation-dissipation for the approximating equations}
\label{ASection3flucdiss}
In this section we consider fluctuation-dissipation based on the Galerkin approximations of $\eqref{Equ_NLS7}$. We will prove that they are globally well-posed on the approximating spaces $E_N$, then we construct a sequence of stationary measures and derive uniform bounds.
 Also, using the projector $P_N$, we recall the notation $E_N=P_NL^2.$\\
The finite-dimensional property of $E^N$ makes all norms well-defined on it to be equivalent. Therefore, unless we need uniformity for an estimate, we may work only with the $L^2-$norm and the result will be automatically valid for the other.

 Set the initial value problem
\begin{align}
\dt u &=\i[(\Delta-1) u-P_N(|u|^{p-1}u)]-\alpha[(1-\Delta)^{s-1}+e^{\rho(\|u\|_{s-})}]u +\sqrt{\al}\eta_N \label{eqN}\\
u|_{t=0} &=w\in E_N. \label{IncondN}
\end{align} 
Here $\rho:\R_+\to \R_+$ satisfies
\begin{equation}\label{Hyprho}
C(\rho,r)e^{\rho(x)}\geq x^r\quad \text{for any $r>1$,}
\end{equation} 
for some constant $C(\rho,r)$ depending only on $\rho$ and $r$ (we can think $\rho$ as a suitable convex function as it will be taken later). 
\subsection{Dissipation rates of the mass and the energy}
In the equation \eqref{eqN}, the mass and the energy given by
\begin{align*}
M(u) &=\frac{1}{2}\int_{\T^3}|u(x)|^2dx,\\
E(u) &=\int_{\T^3}\left(\frac{1}{2}|\nabla u(x)|^2+\frac{1}{2}| u(x)|^2+\frac{1}{p+1}|u(x)|^{p+1}\right)dx.
\end{align*}
interact with  the damping term $\al[(1-\Delta)^{s-1}+e^{\rho(\|u\|_{s-})}]u$. \\
Let us remark that, in this interaction, the quantity $e^{\rho(\|u\|_{s-})}$ is a constant in $x$ and does not "participate" to any integration w.r.t. the $x$ variable. Namely, for a functional $G(u)=\int_{\T^3}g(u)dx$, we have, formally, that
\begin{align*}
G'(u,e^{\rho(\|u\|_{s-})}h(u))=\int_{\T^3}g'(u)e^{\rho(\|u\|_{s-})}h(u)dx=e^{\rho(\|u\|_{s-})}G'(u,h(u)).
\end{align*}

The resulting dissipation rates are formally given by 
\begin{align*}
\al\mathcal{M}(u):=M'(u,\al[(1-\Delta)^{s-1}+e^{\rho(\|u\|_{s-})}] u)=\al[M'(u,(1-\Delta)^{s-1} u)+e^{\rho(\|u\|_{s-})}M'(u,u)]
\end{align*}
 and 
\begin{align*}
\al\mathcal{E}(u):=E'(u,\al[(1-\Delta)^{s-1}+e^{\rho(\|u\|_{s-})}]u)=\al[E'(u,(1-\Delta)^{s-1} u)+e^{\rho(\|u\|_{s-})}E'(u,u)],
\end{align*} 
  respectively. These quantity are well-defined for regular enough solutions. Here, we give some useful properties concerning them. Let us, first, observe that
\begin{align}
M'(u;v) &=(u,v),\\
E'(u;v) &=(-\Delta u+u+|u|^{p-1}u,v).
\end{align}
 We have, using that $P_Nu=u,$ that
\begin{align}
\mathcal{M}(u)= M'(u,[(1-\Delta)^{s-1}+e^{\rho(\|u\|_{s-})}] u) &=\|u\|_{s-1}^2+e^{\rho(\|u\|_{s-}^2)}\|u\|^2.\label{DefmathcalM}
\end{align}
Also, 
\begin{align} 
E'(u,[(1-\Delta)^{s-1}+e^{\rho(\|u\|_{s-})}]u) &=((1-\Delta) u+|u|^{p-1}u,(1-\Delta)^{s-1} u)+e^{\rho(\|u\|_{s-})}((1-\Delta)u+|u|^{p-1}u,u)\nonumber\\
&=\|u\|_s^2+(|u|^{p-1}u,(1-\Delta)^{s-1}u)+e^{\rho(\|u\|_{s-})}(\|u\|_1^2+\|u\|_{L^{p+1}}^{p+1}).\nonumber
\end{align}
For $s=2$, we use an integration by parts to find that
\begin{align}
(|u|^{p-1}u,(1-\Delta)^{s-1}u)\geq 0.
\end{align}
Now, for $s>2,$ using the Agmon's inequality, 
\begin{align*}
|((1-\Delta)^{s-1} u,|u|^{p-1}u)| &= |((1-\Delta)^{\frac{s-1}{2}}u,(1-\Delta)^{\frac{s-1}{2}}|u|^{p-1}u)|\leq C_s\|u\|_{s-1}^2\|u\|_{L^\infty}^{p-1}\leq C_s\|u\|_{s-1}^2\|u\|^{\frac{p-1}{4}}\|u\|_{2}^{\frac{3(p-1)}{4}}.
\end{align*}
Since $s>2$,  we can always find, by performing a Young inequality and the properties of $\rho$ (see \eqref{Hyprho}), a constant $K=K(s,p,\rho)>0$ such that
\begin{align*}
C_s\|u\|_{s-1}^2\|u\|^{\frac{p-1}{4}}\|u\|_{2}^{\frac{3(p-1)}{4}}\leq K+\frac{1}{2}\|u\|^2e^{\rho(\|u\|_{s-})}.
\end{align*}
Overall, one obtains, for all $s\geq 2,$ that
\begin{align}
\mathcal{E}(u)=E'(u,[(1-\Delta)^{s-1}+e^{\rho(\|u\|_{s-})}] u)\geq \|u\|_s^2+\left(\frac{\|u\|_1^2}{2}+\|u\|_{L^{p+1}}^{p+1}\right)e^{\rho(\|u\|_{s-})}-K. \label{DefmathcalE}
\end{align}

\subsection{Globlal well-posedness for the fluctuation-dissipation problems on $E_N$}
Let us introduce the following definition.
\begin{defi}\label{DefGWP}
Let $N\geq 2$. The equation $(\ref{eqN})$ is said to be stochastically globally well-posed on $E_N$ if for all  the following properties hold
\begin{enumerate}
\item  for any random variable $u_0$ in $E_N$ which is independent of $\mathcal{F}_t,$ we have, for almost all $\omega\in\Omega$,
\begin{enumerate}
\item (Existence) there exists $u\in C(\R_+,E_N)$ satisfying \eqref{eqN} and \eqref{IncondN} in which $u_0$ is remplaced by $u_0(\omega).$ We denote the solution by $u^\omega(t,u_0).$
\item (Uniqueness) if $u_1,u_2\in C(\R_+,E_N)$ are two solutions starting at $u_0$ then $u_1\equiv u_2.$ 
\end{enumerate}
\item (Continuity w.r.t. initial data) for almost all $\omega,$ we have
\begin{equation}
\lim_{u_{0}\to u_{0}'}u^\omega(\cdot,u_0)=u^\omega(\cdot,u_0') \ \ \ \text{in $C_tE_N$},
\end{equation}
where $u_0$ and $u_0'$ are deterministic data in $H^s$;
\item the process $(\omega,t)\mapsto u^\omega(t)$ is adapted to the filtration $\sigma(u_0,\mathcal{F}_t)$.
\end{enumerate}
\end{defi}

We claim that the problem \eqref{eqN}, \eqref{IncondN} is stochastically globally well-posed on $E_N$ in the sense of Definition \ref{DefGWP}. The proof of this fact is rather classical and is going to be presented here following the few steps below.
\begin{enumerate}
\item {\bfseries Existence of a global solution.}
Consider the stochastic convolution
\begin{align*}
z(t):=z_\al=\sqrt{\al}\int_0^te^{\i(t-s)(\Delta-1)-\al(1-\Delta)^{s-1}}d\zeta_N(s),
\end{align*} 
this is a well-defined for $\P$-almost all $\omega\in\Omega$; we see that $z$ is the unique solution of the equation
\begin{align}
d z=[\i(\Delta -1)-\al(1-\Delta)^{s-1}]zdt+\sqrt{\al}d\zeta_N,\quad z|_{t=0}=0.\label{eqNz}
\end{align}
We see without difficulties that, for $\P$-almost all $\omega\in\Omega$, $z$ belongs in $C_t(\R_+,C^\infty(E_N))$ (one can apply the It\^o formula to derivatives of \eqref{eqNz}). 

Now any $\omega \in \Omega$ such that $z^\omega$ belongs to $C_t(\R_+,C^\infty(E_N))$, we set the problem
\begin{align}
\dt v &=\i[(\Delta-1) v-P_N|v+z|^{p-1}(v+z)]-\al[(\Delta-1)^{s-1}v+e^{\rho(\|v+z\|_{s-})}(v+z)],\label{eqNv}\\
 v|_{t=0} &=u_0\in E_N.\nonumber
\end{align}
Since the map $E_N\to E_N:\ \ v\mapsto \i[(\Delta-1) v-P_N|v+z|^{p-1}(v+z)]-\al[(\Delta-1)^{s-1}v+e^{\rho(\|v+z\|_{s-})}(v+z)]$ is smooth, thanks to the classical Cauchy-Lipschitz theorem, the problem \eqref{eqNv} has a local in time smooth solution. We see that this solution is in fact global in time by using the Proposition \ref{prop_control_M_E}  below. Now, observe that the sum $v+z$ is a solution to  \eqref{eqN}, \eqref{IncondN}. Also, $u\in C_t(\R_+,C^\infty)$ since $z\in C_t(\R_+,C^\infty)$ and $F\in C^\infty(E_N,E_N).$
\begin{prop}\label{prop_control_M_E}
The local solution $v$ constructed above exists globally in time, $\P-$almost surely.
\end{prop}
\begin{proof}
Let us compute the derivative of $\frac{\|v\|^2}{2}$ and use the equation \eqref{eqNv} and \eqref{Hyprho}, we obtain
\begin{align*}
\frac{d}{dt}\left[\frac{\|v\|^2}{2}\right] &=-(z,\i|v+z|^{p-1}(v+z))-\al[\|v\|_{s-1}^2+\|v\|^2e^{\rho(\|v+z\|_{s-})}]+\alpha(v,z)e^{\rho(\|v+z\|_{s-})}\\
&\leq \frac{\|z\|^2}{2\al}+\frac{\al\|v+z\|^{2p}}{2}-\al\|v\|^2e^{\rho(\|v+z\|_{s-})}+\al\left[\frac{\|v\|^2}{2}+\frac{\|z\|^2}{2}\right]e^{\rho(\|v+z\|_{s-})}\\
&\leq \frac{\|z\|^2}{2\al}+\frac{\al}{2}\left[C(\rho,p)+\|z\|^2-\|v\|^2\right]e^{\rho(\|v+z\|_{s-})}.
\end{align*} 
Now we have that for $\P-$almost all $\omega\in\Omega$ for all $T$, there is a constant $C_\al(\omega,T)$ such that 
\begin{align}
\sup_{t\in[0,T]}\|z(\omega,t)\|\leq C_\al(\omega,T).\label{Controlonz}
\end{align}
Now, for a fixed $\omega$ such that \eqref{Controlonz} holds, fix any $T>0$. Let $t\in[0,T]$, we have the following two complementary scenarios:
\begin{enumerate}
\item either $\|v(\omega,t)\| \leq C(\rho,p)+C_\al(\omega,T)$
\item or $\|v(\omega,t)\| > C(\rho,p)+C_\al(\omega,T)$. In this case $\left[C(\rho,p)+\|z\|^2-\|v\|^2\right]e^{\rho(\|v+z\|_{s-})}<0$ and
\begin{align*}
\frac{d}{dt}\left[\frac{\|v\|^2}{2}\right]\leq \frac{\|z\|^2}{2\al}\leq C^1_\al(\omega,T).
\end{align*}
\end{enumerate}
Therefore
\begin{align*}
\frac{d}{dt}\left[\frac{\|v\|^2}{2}\right]\leq C^2_\alpha(\omega,T)\quad\text{for any $t\in[0,T]$}.
\end{align*}
Overall
\begin{align*}
\sup_{t\in[0,T]}\|v\|^2\leq \|u_0\|^2+C^3_\alpha(\omega,T)\quad\text{for any $T>0$}.
\end{align*}
In conclusion, we obtain the claim by using an iteration argument.
\end{proof}

\item {\bfseries Uniqueness and continuity.}
For a fixed $\omega\in \Omega,$ let $u_i,\ i=1,2$ two solutions to \eqref{eqN} starting at $u_{0,i},\ i=1,2,$ respectively. Let $w:=u_1-u_2$, and $F(u)=\i[(\Delta-1) u-P_N(|u|^6u)]-\alpha[(1-\Delta)+e^{\rho(\|u\|_{s-})}]u$, $F$ is clearly in $C^\infty(E^N\to E^N).$ Using the difference of the corresponding two equations, we see readily that $w$ satisfies the equation
\begin{align*}
\dt w=F(u_1)-F(u_2)=w\int_0^1F'(su_1+(1-s)u_2)ds.
\end{align*}
Taking the inner product with $w$, we have
\begin{align*}
\dt\|w\|^2\leq \|w\|^2\sup_{x\in\T^3}\int_0^1|F'(su_1+(1-s)u_2)|ds.
\end{align*}
Using the Gronwall lemma, we obtain
\begin{align*}
\|u_1(t)-u_2(t)\|\leq \|u_{0,1}-u_{0,2}\|\exp\left(\frac{t}{2}\sup_{x\in\T^3,t\in[0,T]}\int_0^1|F'(su_1+(1-s)u_2)|ds\right).
\end{align*}
This estimate implies uniqueness  in $C_tL^2$, as well as continuity with respect to the initial datum in any Sobolev type norm of $E_N$, because of the finite-dimensionality.
\item {\bfseries Adaptation.} It is clear that $z$ is adapted to $\mathcal{F}_t$, since $v$ is constructed by a fixed point argument, then it is adapted to $\sigma(u_0,\mathcal{F}_t).$ We obtain the claim.
\end{enumerate}
Let us denote by $u_\al(t,u_0)$ the unique solution to \eqref{eqN}, \eqref{IncondN}. 
\subsection{Stationary solutions and uniform estimates}
\subsubsection{A Markov framework.}
Let us define the transition probability
\begin{align*}
T_{\al,N}^t(w,\Gamma)=\P(u_\al(t,P_Nw)\in\Gamma)\quad w\in L^2,\ \ \Gamma\in \text{Bor}(L^2),\ \ t\geq 0,
\end{align*}
and  define the Markov semi-groups
\begin{align*}
\mathfrak{P}_{\al,N}^{t}f(v)&=\int_{L^2}f(w)T_{\al,N}^t(v,dw)\ \ \ L^\infty(L^2;\R)\to L^\infty(L^2;\R),\\
\mathfrak{P}_{\al,N}^{t*}\lambda(\Gamma)&=\int_{L^2}\lambda(dw)T_{\al,N}^t(P_Nw,\Gamma)\ \ \  \mathfrak{p}(L^2)\to \mathfrak{p}(L^2).
\end{align*}
Since the solution $u(t,u_0)$ is continuous in $u_0$, the Markov semi-group $\mathfrak{P}_{\al,N}^{t}$ is Feller: for any $t\geq 0,\ \mathfrak{P}_{\al,N}^{t}C_b(L^2)\subset C_b(L^2).$ Hence we can consider it as acting on this space.
\subsubsection{Statistical estimates of the flow.}
Set the truncated constants
\begin{align}
A_{s,N}=\sum_{|m|\leq N}\lambda_m^s|a_m|^2.\label{ConstantAsN}
\end{align}
Of course, these constants are bounded respectively by
\begin{align}
A_s=\sum_{m\in\N}\lambda_m^s|a_m|^2,\label{ConstantAs}
\end{align}
that we assume to be finite for $s=1.$ Also we have then the obvious convergence $A_{s,N}\to A_s$, as $N\to\infty$ for $s\leq 1.$
\begin{prop}\label{prop_est_prob_Mass}
Let $u_0$ be a random variable in $E_N$ independent  of $\mathcal{F}_t$ such that $\E M(u_0)\pp \infty.$ Let $u$ be the solution to \eqref{eqN} starting at $u_0.$
Then we have
\begin{align}
\E M(u)+\alpha\int_0^t\E\mathcal{M}(u)d\tau &=\E M(u_0)+\frac{\alpha A_{0,N}}{2}t.\label{est_esp_L2}
\end{align}
\end{prop}
\begin{proof}
We apply the finite-dimensional It\^o formula  to the functional $M(u):$  
\begin{align*}
dM(u) &=M'(u,du)+\frac{\alpha}{2}\sum_{|m|\leq N} a_m^2M''(u;e_m,e_m)dt.
\end{align*}
Now, using the fact that $M'(u,\i(\Delta u-|u|^{p-1}u))=0$ and \eqref{DefmathcalM}, we have
\begin{align*}
M(u,du)&=-\al\mathcal{M}(u)dt+\sqrt{\al}\sum_{|m|\leq N}a_m(u,e_m)d\beta_m.
\end{align*}
On the other hand,
\begin{align*}
M''(u;e_m,e_m) &=\|e_m\|^2=1.
\end{align*}
Then, after integration in $t$ and taking the expectation, we arrive at the \eqref{est_esp_L2}. 
\end{proof}

\begin{prop}\label{prop_est_prob_Energy}
Let $u_0$ be a random variable in $E_N$ independent  of $\mathcal{F}_t$. Suppose that $\E E(u_0)<\infty$, then we have
\begin{align}
\E E(u)+\alpha\int_0^t\E\mathcal{E}(u)d\tau &\leq \E E(u_0)+\frac{\alpha}{2}\left(A_{1,N}t+A_{0,N}(2\pi)^{-3}\int_0^t\E\|u\|_{L^{p-1}}^{p-1}d\tau\right),\label{est_espEnergy}
\end{align}
where $u$ is the solution to \eqref{eqN} starting at $u_0.$
\end{prop}
\begin{proof}
 We apply the It\^o's formula to $E(u)$, and use the fact that $E'(u,\i(\Delta u-|u|^{p-1}u))=0$ and \eqref{DefmathcalE}, we obtain 
\begin{align*}
E(u)+\alpha \int_0^t\mathcal{E}(u)d\tau &\leq E(u_{0})+\frac{\alpha}{2}\left(A_{1,N}t+\sum_{|m|\leq N}a_m^2\int_0^t(|u|^{p-1};e_m,e_m)d\tau\right)\\
&+\sqrt{\alpha}\sum_{|m|\leq N}a_m\int_0^t\left(\lambda_m(u,e_m)+(|u|^{p-1}u,e_m)\right)d\beta_m(\tau).
\end{align*}
Taking the expectation, we obtain
\begin{align*}
\E E(u)+\alpha \E\int_0^t\mathcal{E}(u)d\tau &\leq \E E(u_{0})+\frac{\alpha}{2}\left(A_{1,N}t+\E\sum_{|m|\leq N}a_m^2\int_0^t(|u|^{p-1};e_m,e_m)d\tau\right)\\
&\leq \E E(u_{0})+\frac{\alpha}{2}\left(A_{1,N}t+A_{0,N}(2\pi)^{-3}\E\int_0^t\|u\|_{L^{p-1}}^{p-1}d\tau\right).
\end{align*}
The proof is finished.
\end{proof}

We can see without difficulties the following statement:
\begin{prop}
The solution $z_\al$ to \eqref{eqNz} satisfies the estimate
\begin{align}\label{estimeeL2puissancepZ}
\E\|z_\al(t)\|^{2p}\leq \alpha C(p,A_0,t),
\end{align}
where, $C(p,A_0,t)$ does not depend on $\alpha,$ and $p\geq 1$.
\end{prop}
\begin{cor}
The solution $z_\al$ to \eqref{eqNz} satisfies the estimate
\begin{equation}\label{Doobfinal}
\E\sup_{t\in [0,T]}\|z_\al(t)\|^{2p}\leq C(p,A_0,T)\al ,
\end{equation}
where, $C(p,A_0,T)$ does not depend on $\alpha,$ and $p\geq 1$.
\end{cor}
\begin{proof}
We have that $z_\al$ is a martingale adapted to $\mathcal{F}_s$, thanks to the well known properties of the It\^o integral. Since the function $u\mapsto \|u\|^{2p}, \ p\geq 1$ is convex, then $\|z_\al\|^{2p}$ is a submartingale. Then by the Doob inequality,
\begin{align}
\E\sup_{t\in [0,T]}\|z_\al(t)\|^{2p}\leq C_p\E\|z_\al(T)\|^{2p}.\label{Doob}
\end{align}
We finish the proof after a use of the estimate \eqref{estimeeL2puissancepZ}.
\end{proof}

\subsubsection{Existence of stationary measures and uniform  bounds.}
\begin{thm}\label{thmStatMes}
For any $N\geq 2$ and any  $\alpha\in (0,1)$, there is an stationary measure $\mu_{\alpha,N}$ to $(\ref{eqN})$ concentrated on $H^3$. Moreover, we have the following estimates
\begin{align}
\int_{L^2}\mathcal{M}(u)\mu_{\alpha,N}(du) &=\frac{A_{0,N}}{2} \leq \frac{A_0}{2}, \label{est_StatMeas_mathcalM}\\
\int_{L^2}\mathcal{E}(u)\mu_{\alpha,N}(du) &\leq C.\label{est_StatMeas_mathcalE}
\end{align}
where $C$ does not depend on $\alpha$ and $N.$
\end{thm}
\begin{proof}
{\bfseries Existence of stationary measures.}
Let $B_R$ be the ball of $H^1$ with center $0$ and radius $R$. We have, with the use of the Chebyshev inequality, that 
\begin{align*}
\frac{1}{t}\int_0^t\mathfrak{P}_{\al,N}^{\tau*}\lambda(B_R^c)d\tau &=\int_{L^2}\lambda(dw)\frac{1}{t}\int_0^t T_{\al,N}^{\tau}(w,B_R^c)d\tau\\
&\leq \int_{L^2}\lambda(dw)\frac{1}{t}\int_0^t \frac{\E\|\phi_\alpha (t)w\|_1^2}{R^2}d\tau\\
&\leq \int_{L^2}\lambda(dw)\frac{\E\|w\|_1^2+\frac{\alpha A_{0,N}}{2}t}{\alpha tR^2}\\
&\leq \frac{1}{R^2}\left(\frac{\E_\lambda\|w\|^2_1}{\alpha t}+\frac{A_{0,N}}{2}\right).
\end{align*}
Choose $\lambda$ to be the Dirac measure concentrated at $0,$ $\delta_0.$ Then we obtain
\begin{align*}
\frac{1}{t}\int_0^t\mathfrak{P}_{\al,N}^{\tau*}\delta_0(B_R^c)d\tau &\leq \frac{A_{0,N}}{2R^2}.
\end{align*}
Whence follows the compactness of the the family $\left\{\frac{1}{t}\int_0^t\mathfrak{P}_{\al,N}^{\tau*}\lambda d\tau, \ t\geq 0\right\}$ on $E_N.$\\
Let $\mu_{\alpha,N} \in \mathfrak{p}(E_N)$ be an accumulation point at $\infty$, that is,
\begin{align*}
\mu_{\alpha,N}=\lim_{n\to\infty}\frac{1}{t_n}\int_0^{t_n}\mathfrak{P}_{\al,N}^{\tau*}\delta_0 d\tau.
\end{align*} 
The well known Bogoliubov-Krylov argument states that $\mu_{\alpha,N}$ is stationary for $(\ref{eqN}).$

{\bfseries Estimates of the stationary measures.}
Denote by $\lambda_t$ the measure $\frac{1}{t}\int_0^t\mathfrak{P}_{\al,N}^{\tau*}\delta_0 d\tau$. We have  that, using \eqref{est_esp_L2}, that 
\begin{align*}
\int_{L^2}\mathcal{M}(v)\chi_R(\|v\|)\lambda_t(dv)\leq \int_{L^2}\mathcal{M}(v)\lambda_t(dv)\leq\frac{A_{0,N}}{2}.
\end{align*}
Now, since $v\mapsto \mathcal{M}(v)\chi_R(\|v\|)$ is continuous bounded on $L^2$, we have by passing to the limit on a subsequence $t_n\to\infty$
\begin{align*}
\int_{L^2}\mathcal{M}(v)\chi_R(\|v\|)\mu(dv)\leq\frac{A_{0,N}}{2}.
\end{align*}
We know use the Fatou's lemma (under the limit $R\to \infty$) to obtain
\begin{align*}
\int_{L^2}\mathcal{M}(v)\mu(dv)\leq\frac{A_{0,N}}{2}.
\end{align*}
In particular
\begin{align}
\int_{L^2}\|v\|^q\mu(dv)\leq\frac{A_{0,N}}{2}\quad \text{for any $q\geq 1$},\label{Require0}
\end{align}
and using finite-dimensionality (with \eqref{Require0}), we remark that
\begin{align}
\int_{L^2}E(v)\mu(dv)<\infty.\label{Require2}
\end{align}
Therefore, \eqref{Require0} provides the  requirement of Proposition \ref{prop_est_prob_Mass}, then we obtain the identity \eqref{est_esp_L2}. But since the measure $\mu_\al$ is stationary, we arrive at the identity \eqref{est_StatMeas_mathcalM}. \\
To establish the estimate \eqref{est_StatMeas_mathcalE}, let us observe that \eqref{Require2} implies the condition of Proposition \ref{prop_est_prob_Energy}. Therefore we have the estimate \eqref{est_espEnergy}. Again, thanks to the stationarity of $\mu_\al$ and the fact that $A_{s,N}\leq A_s$, we obtain \eqref{est_StatMeas_mathcalE} with the required constant $C$. 
\end{proof}

Let us give an additional estimate for the measures $\mu_{\al,N}.$
Let $\chi$ be a $C^\infty$ function having value $1$ on $[0,1]$ and $0$ on $[2,\infty).$ Clearly the function $\chi$ and its derivative are bounded, we can take a universal constant that bounds $\chi$ and its first two derivative, so without loss of generality we can consider this constant to be $1$. Set $\chi_R(x)=\chi(\frac{x}{R})$. We then have the following estimates on the derivatives of $\chi_R:$ 
\begin{align}
|\chi_R^{(m)}(x)| &\leq R^{-m}.
\end{align}
We have that
\begin{prop}
For any $R>0$, the following estimate holds
\begin{align}
\int_{L^2}\mathcal{M}(u)(1-\chi_R(\|u\|^2))\mu_{\al,N}(du)\leq C_2R^{-1}.\label{est_StatMeas_MR}
\end{align}
where $C_2$ is independent of $(\al,N)$.
\end{prop}
\begin{proof}
Let $F_R(u)=\|u\|^2(1-\chi_R)(\|u\|^2)$. Applying Ito's formula we see the following
\begin{align*}
dF_R &+2\alpha \mathcal{M}(u)(1-\chi_R(\|u\|^2)) =-2\alpha\mathcal{M}(u)\chi'_R(\|u\|^2)\\
&+\frac{\al}{2}\left((1-\chi_R(\|u\|^2))A_0^N+2\chi_R'(\|u\|^2)\sum_{|m|=0}^Na_m^2(u,e_m)^2+\|u\|^2\left[A_0^N\chi_R'(\|u\|^2)+\chi_R''(\|u\|^2)\sum_{|m|=0}^Na_m^2(u,e_m)^2\right]\right).
\end{align*}
Let us use the invariance and \eqref{est_StatMeas_mathcalM} to get
\begin{align}
\E\mathcal{M}(u)(1-\chi_R(\|u\|^2))\leq A_0^N\E(1-\chi_R(\|u\|^2))+\frac{C(A_0)}{R}.
\end{align}
Now using the Markov inequality and \eqref{est_StatMeas_mathcalM}, we have
\begin{align}
\int_{L^2}(1-\chi_R(\|u\|^2))\mu_{\al,N}\leq C R^{-1} 
\end{align}
where $C$ is independent of $(\al, N)$. Overall,
\begin{align}
\int_{L^2}\mathcal{M}(u)(1-\chi_R(\|u\|^2))\mu_{\al,N}\leq C R^{-1},
\end{align}
which is the claim.
\end{proof}

\section{Inviscid limit towards the approximating NLS-7 equations}
\label{ASection4Invisc}
We consider now the truncated NLS equations
\begin{align}
\dt u &=\i[(\Delta-1)u+P_N(|u|^{p-1}u)],\label{NLS-7N}\\
u|_{t=0} &=u_0\in E_N.
\end{align}
Using the preservation of the $L^2-$norm, we see that the local solutions constructed in Proposition \ref{propLWP} are in fact  global.  Uniqueness and continuity follow, through usual methods, from the regularity of the non-linearity, we then obtain global well-posedness. Define the associated global flow $\phi_{N}(t):E_N\to E_N,\ \ u_0\mapsto \phi_N^tu_0$, where $\phi_N^t(u_0)=:u(t,u_0)$ represents the solution to \eqref{NLS-7N} starting at $u_0$. Let us set the corresponding Markov groups 
\begin{align*}
\Phi_N^tf(v) &=f(\phi_{N}^t(P_Nv));\quad C_b(L^2)\to C_b(L^2),\\
\Phi_{N}^{t*}\lambda(\Gamma) &=\lambda({\phi_N^{-t}}(\Gamma)); \quad \mathfrak{p}(L^2)\to \mathfrak{p}(L^2).
\end{align*} 
From the estimate \eqref{est_StatMeas_mathcalM}, we have the weak compactness of any sequence $(\mu_{\al_k,N})$ with respect to the topology of $H^{1-\epsilon},$ therefore there exists a subsequence $(\mu_{\al_{k},N})=:(\mu_{k,N}),$ converging to a measure $\mu_N$ on $L^2$. We have the following
\begin{prop}
Let $N\geq 2$, the measure $\mu_N$ is invariant under $\phi_N(t)$ and satisfies the estimates
\begin{align}
\int_{L^2}\mathcal{M}(u)\mu_N(du) &=\frac{A_{0,N}}{2}\leq \frac{A_{0}}{2}, \label{est_muN_M}\\
\int_{L^2}\mathcal{E}(u)\mu_N(du) &\leq C_1, \label{est_muN_E}\\
\int_{L^2}\mathcal{M}(u)(1-\chi_R(\|u\|^2))\mu_N(du) &\leq C_2R^{-1}.\label{est_muN_MR}
\end{align} 
where $C_1$ and $C_2$ are independent of $N$. 
\end{prop}
Below, the subscript $k$ stands for $\al_k$. We do this abuse of notation to simplify the formulas.
\begin{proof}
\begin{enumerate}
\item {\bfseries Estimates.}
The estimates \eqref{est_muN_E} and \eqref{est_muN_MR} follow respectively from \eqref{est_StatMeas_mathcalE} and \eqref{est_StatMeas_MR}  and the lower semicontinuity of $\mathcal{E}(u)$ and $\mathcal{M}(u)$. Now let us prove \eqref{est_muN_M}: let $\chi_R$ be a bump function on $\R$ having the value $1$ on $[0,1]$ and the value $0$ on $[2,\infty),$ we write
\begin{align*}
\frac{A_{0,N}}{2}-\int_{L^2}(1-\chi_R(\|u\|^2))\mathcal{M}(u)\mu_{k,N}(du)\leq \int_{L^2}\chi_R(\|u\|^2)\mathcal{M}(u)\mu_{k,N}(du)\leq \frac{A_{0,N}}{2}.
\end{align*}
Now, using \eqref{est_StatMeas_MR},
\begin{align*}
\frac{A_{0,N}}{2}-C_2R^{-1}\leq \int_{L^2}\chi_R(\|u\|^2)\mathcal{M}(u)\mu_{k,N}(du)\leq \frac{A_{0,N}}{2}.
\end{align*}
It remains to pass to the limits $k\to\infty$, then $R\to\infty$ to arrive at the claim. 
\item {\bfseries Invariance.} It suffices to show the invariance under $\phi_N^t,\ t>0$. Indeed
For $t<0,$ we have, using the invariance for positive times, that
\begin{align*}
\mu_N(\Gamma)=\mu_N(\phi_N^{t}\Gamma)=\mu_N(\phi_N^{2t}\phi_N^{-t}\Gamma)=\mu_N(\phi_N^{-t}\Gamma),
\end{align*}
which is the needed property. Now the proof of the invariance for positive times is summarized in the following diagram
$$
\hspace{10mm}
\xymatrix{
  \mathfrak{P}_{k,N}^{t*}\mu_{k,N} \ar@{=}[r]^{(I)} \ar[d]^{(III)} & \mu_{k,N} \ar[d]^{(II)} \\
    \Phi_N^{t*}\mu_N \ar@{=}[r]^{(IV)} & \mu_N
  }
$$
\end{enumerate}
The equality $(I)$ represents the stationarity  of $\mu_{k,N}$ under $\mathfrak{P}_{k,N}^{t}$, $(II)$ is the weak convergence of $\mu_{k,N}$ towards $\mu_N$. The equality $(IV)$ represents the (claimed) invariance of $\mu_N$ under $\phi_N$, that will follow once we prove the convergence $(III)$ in the weak topology of $L^2.$ To this end, let $f: L^2\to\R$ be a Lipschitz function that is also bounded by $1.$ We have
\begin{align*}
(\mathfrak{P}_{k,N}^{t*}\mu_{k,N},f)-(\Phi_N^{t*}\mu_N,f)&=(\mu_{k,N},\mathfrak{P}_{k,N}^tf)-(\mu_N,\Phi_N^tf)\\
&=(\mu_{k,N},\mathfrak{P}_{k,N}^tf-\Phi_N^tf)-(\mu_N-\mu_{k,N},\Phi_N^tf)\\
&=A-B.
\end{align*}
Since $\Phi_N^t$ is Feller, we have that $B\to 0$ as $k\to\infty.$ Now, using the boundedness property of $f$, we have
\begin{align*}
|A|\leq \int_{B_R(L^2)}|\Phi_N^tf(u)-\mathfrak{P}_{k,N}^tf(u)|\mu_{k,N}(du)+2\mu_{k,N}(L^2\backslash B_R(L^2))=:A_1+A_2.
\end{align*}
Here $C_f$ is the Lipschitz constant of $f$ and $u_k(t,P_Nu)$ is the solution to \eqref{eqN} at time $t$ and starting from $P_Nu.$ Now from \eqref{est_StatMeas_mathcalM}, we have
\begin{align*}
A_2\leq \frac{C}{R^2}.
\end{align*}
To treat the term $A_1$, let us consider the set $$S_r=\left\{\omega\in\Omega|\ \ \max\left(|\sqrt{\al_k}\sum_{|m|\leq N}\lambda_m\int_0^t(u,e_m)d\beta_m|, \|z_k\|\right)\leq r\sqrt{\al_kt}\right\}\quad r>0,$$ we have the following statement.
\begin{nem}\label{LemConvUnifInvis}
We have that, for any $R>0,$ any $r> 0,$
\begin{align}
\sup_{u\in B_R(L^2)}\E(\|\phi_N^tP_Nu- u_k(t,P_Nu)\|\Bbb 1_{S_r})\to 0,\quad as\ k\to\infty.
\end{align}
\end{nem}
Now let us split $A_1$, and use the Lispschitz and boundedness properties of $f$
\begin{align*}
A_1\leq C_f\int_{B_R}\E\|\phi_N^tP_Nu-u_k(t,P_Nu)\|\Bbb 1_{S_r}\mu_{k,N}(du)+2\int_{B_R}\E(\Bbb 1_{S_r^c})\mu_{k,N}(du) =:A_{1,1}+A_{1,2}.
\end{align*}
It follows from the Lemma \ref{LemConvUnifInvis} above that, for any fixed $R>0$ and $r>0,$ $\lim_{k\to\infty}A_{1,1}=0.$\\
Now, it follows from the classical It\^o isometry and \eqref{est_StatMeas_mathcalM} that
\begin{align*}
\E|\sqrt{\al_k}\sum_{|m|\leq N}\lambda_m\int_0^t(u,e_m)d\beta_m|^2=\al_k\sum_{|m|\leq N}\lambda^2_m\int_0^t(u,e_m)^2dt\leq \al_k A_0\E\|u\|^2\leq C\al_k,
\end{align*}
where $C$ does not depend on $k.$ Also, from \eqref{estimeeL2puissancepZ},
\begin{align*}
\E\|z_k\|^2\leq C\al_k,
\end{align*}
where $C$ is independent of $k.$ Therefore, using the Chebyshev inequality, we have
\begin{align*}
\E(\Bbb 1_{S^c_r})=\P\left\{\omega\ | \  \max\left(\left|\sqrt{\al_k}\sum_{|m|\leq N}\lambda_m\int_0^t(u,e_m)d\beta_m\right|, \|z_k\|\right)\geq r\sqrt{\al_kt} \right\}\leq \frac{C\al_k}{r^2\al_k}=\frac{C}{r^2}.
\end{align*}
Passing to the limits $k\to\infty,\ R\to\infty,\ r\to\infty$ (respecting this order), we obtain $(III)$, and hence $(IV).$
\end{proof}
\begin{proof}[Proof of Lemma \ref{LemConvUnifInvis}]
Set $w_k=u-v_k:=\phi_N^tP_Nu_0-v_k(t,P_Nu),$ where $v_k(t,P_Nu_0)$ is the solution of \eqref{eqNv}, with $\alpha =\al_k$ and that starts from $P_Nu_0.$ We recall that $u_k=v_k+z_k,$ where $z_k$ solves the problem \eqref{eqNz} with $\al=\al_k.$ Now, thanks to \eqref{Doobfinal}, we have that $\E\|z_k\|^2\to 0$ as $k\to\infty.$
Therefore, it suffices to show that
\begin{align}
\sup_{u_0\in B_R(L^2)}\E(\|w_k\|\Bbb 1_{S_r})\to 0,\quad as\ k\to\infty.
\end{align}
to complete the proof of the Lemma \ref{LemConvUnifInvis}.\\ Let us take the difference between the equation \eqref{NLS-7N} and $\eqref{eqNv}$:
\begin{align*}
\dt w_k&= \i[(\Delta-1) w_k-P_N(w_kf_{p-1}(u,v_k))]+\i P_N(g_{p-1}(u,v_k,z_k)z_k)\\
&-\al_k[(1-\Delta)^{s-1}v_k+e^{\rho(\|v_k+z_k\|_{s-})}(v_k+z_k)],
\end{align*}
where $f_{p-1}$ and $g_{p-1}$ are polynomial of degree $p-1$ in the given variables. We observe that $|v_k+z_k|^{p-1}(v_k+z_k)-|u|^{p-1}u=|v_k|^{p-1}v_k-|u|^{p-1}u+z_kg_{p-1}(v_k,z_k)=wf_{p-1}(u,v_k)+z_kg_{p-1}(v_k,z_k).$ \\
Taking the inner product with $w$, we obtain
\begin{align*}
\dt\|w_k\|^2 &\leq 2\|w_k\|^2(1+\lambda_N^2+\|f_{p-1}(u,v_k)\|_{L^\infty_{t,x}})+2\|z_k\|^2\|g_{p-1}(v_k,z_k)\|^2_{L^\infty_{t,x}}\\
&+\al_kC_0(N)\|w_k\|(\|v_k\|+\|z_k\|)e^{\rho(c(N)(\|v_k+z_k\|))}\\
&\leq C_1(N)\|w_k\|^2(1+\lambda_N^2+\|u\|_{L^\infty_{t}L^2_x}^{p-1}+\|v_k\|^{p-1}_{L^\infty_tL^2_x})+C_2(N)\|z_k\|^2\left(\|v_k\|^{2p-2}_{L^\infty_tL^2_x}+\|z_k\|^{2p-2}_{L^\infty_tL^2_x}\right)\\
&+\alpha_kC_3(N)(\|v_k\|^2+\|z_k\|^2)e^{2\rho(c(N)(\|v_k\|+\|z_k\|))}.
\end{align*} 
Using the Gronwall lemma, the fact that $w_k(0)=0$ and using \eqref{Hyprho}, we arrive at 
\begin{align}
\|w_k(t)\|^2\leq C_4(N)e^{C_2(N)\int_0^t(1+\lambda_N^2+\|u\|_{L^\infty_{t}L^2_x}^{p-1}+\|v_k\|^{p-1}_{L^\infty_tL^2_x})d\tau}\left(\int_0^t\|z_k\|^2d\tau+\al_kt\right)\left[1+e^{\rho(c(N)(\|v_k\|_{L^\infty_t L^2_x}+\|z_k\|_{L^\infty_t L^2_x}))}\right]\label{ConvGENERALEst}
\end{align}
 and the estimate \eqref{Doobfinal}, we have that, up to a subsequence, 
\begin{align*}
\lim_{k\to\infty}\sup_{t\in[0,T]}\|w_k\|=0,\ \ \P-almost\ surely.
\end{align*}
Now, writing the It\^o formula for $\|u\|^2$, we have
\begin{align*}
\|u_k\|^2+2\al_k\int_0^t\mathcal{M}(u_k)d\tau =\|P_Nu_0\|^2+\al_k \frac{A_{0,N}}{2}t+2\sqrt{\al_k}\sum_{|m|\leq N}\lambda_m\int_0^t(u_k,e_m)d\beta_m.
\end{align*}
Therefore, recalling that $\al_k\leq 1,$ we have that, on the set $S_r,$
\begin{align*}
\|u_k\|^2\leq \|P_Nu_0\|^2+C(r,N)t,
\end{align*}
where $C(r,N)$ does not depend on $k.$ Hence we see that, on $S_r$,
\begin{align}
\|w_k\|\leq \|v_k\|+\|z_k\|\leq \|u_k\|+2\|z_k\|\leq \|u_0\|+3C(r,N)t.
\end{align}
In particular, we have the following two estimates:
\begin{align}
\sup_{u_0\in B_R}\|w_k\|_{L^\infty_tL^2_x}\Bbb 1_{S_r}\leq R+3C(r,N)T\label{ControlTronqueeStoch}
\end{align}
\begin{align}
\sup_{k\geq 1}\sup_{u_0\in B_R}\|w_k\|_{L^\infty_tL^2_x}\Bbb 1_{S_r}\leq R+3C(r,N)T.\label{CONVUnifin_k_bound}
\end{align}
Hence coming back to \eqref{ConvGENERALEst} and using the (deterministic) conservation $\|u(t)\|=\|P_N u_0\|$ and the estimate \eqref{ControlTronqueeStoch}, we obtain
\begin{align}
\sup_{u_0\in B_R}\|w_k\|_{L^\infty_tL^2_x}^2\Bbb 1_{S_r}\leq A(R,N,r,T)\|z_k\|_{L^1_tL^2_x}.
\end{align}
Therefore, using again the bound \eqref{Doobfinal}, we obtain the almost sure convergence $\|z_k\|\to 0$ (as $k\to\infty$, up to a subsequence), we obtain then the almost sure convergence
\begin{align}
\lim_{k\to\infty}\sup_{u_0\in B_R}\|w_k\|_{L^\infty_tL^2_x}^2\Bbb 1_{S_r}=0.
\end{align}
Now, taking into account the bound \eqref{CONVUnifin_k_bound}, we can then use the Lebesgue dominated convergence theorem to obtain
\begin{align*}
\E\sup_{u_0\in B_R}\|w_k\|_{L^\infty_tL^2_x}\Bbb 1_{S_r}\to 0,\quad as\ k\to\infty.
\end{align*}

Now, for $u_0\in B_R,$ we have
\begin{align*}
\|w_k(t,P_Nu_0)\|\Bbb 1_{S_r}\leq \sup_{u_0\in B_R}\|w_k(t,P_Nu_0)\|\Bbb 1_{S_r},
\end{align*}
then
\begin{align*}
\E\|w_k(t,P_Nu_0)\|\Bbb 1_{S_r}\leq \E \sup_{u_0\in B_R}\|w_k(t,P_Nu_0)\|\Bbb 1_{S_r},
\end{align*}
and finally,
\begin{align*}
\sup_{u_0\in B_R}\E\|w_k(t,P_Nu_0)\|\Bbb 1_{S_r}\leq \E \sup_{u_0\in B_R}\|w_k(t,P_Nu_0)\|\Bbb 1_{S_r}.
\end{align*}
The proof is finished.
\end{proof}
\section{Statistical ensemble for NLS-7 and almost sure GWP}
\label{ASection5StatensemGWP}
In this section, we consider the Schr\"odinger equations
\begin{align}
\dt u &=\i[(\Delta-1)u-|u|^{p-1}u] \label{NLS7totalequ}\\
u|_{t=0} &=u_0. \label{NLS7totaldata}
\end{align}
We follow closely the arguments of \cite{bourg94} (see also \cite{tzvNLS06,tzvNLS}) in the construction of an statistical ensemble for the NLS equations \eqref{NLS7totalequ}. We show that on this set, the equation  is globally well-posed, and the probability measure used in the construction  is left invariant under the flow that has been established. In contrast with the `Gaussianity' of the measures in  \cite{bourg94}, here we do not have many information about the relations between the approximating measures and the limiting measure. Therefore in establishing the statistical ensemble, we need additional tools; that is why we introduce the restricted measures that, combined with the Skorokhod representation theorem, allow to defined an `almost sure limiting' set whose elements can be compared with finite-dimensional data for which the associated solutions are controlled. These controls are inherited by the infinite-dimensional solutions living on the limiting set by the use of the Bourgain iteration procedure \cite{bourg94}: we then obtain global wellposedness on the constructed set.   

In the sequel we consider any function increasing concave function $\xi:\R_+\to\R_+$ and take the function $\rho$ to be $\rho(x)=3\xi^{-1}(x)$, where $\xi^{-1}$ is the inverse of $\xi$. Remark that $\xi^{-1}:\R_+\to\R_+$ is an increasing convex function. Then we remark that $e^{3\xi^{-1}}$ satisfies \eqref{Hyprho}. Also,
\begin{equation}
\xi(x)\leq C_\xi x,
\end{equation}
where $C_\xi$ depends only on $\xi.$\\
To proof Theorem \ref{ThmPrincipal}, we can consider exterior of $B_a$, where $B_a=\{u\in L^2\ | \ \|u\|\leq a\}$, the number $a>0$ being arbitrary small, invoking the conservation of the $L^2$-norm under the NLS equation \eqref{NLS7totalequ}.  Let us set $L^2_a=L^2\backslash B_a$, and $E^N_a=\{u\in E_N^a\ |\ \|u\|\geq a\}$.
Notice that
\begin{align}
\int_{E^N_a}e^{3\xi^{-1}(\|u\|_{s-})}\mu_N(du)=\int_{L^2_a}e^{3\xi^{-1}(\|u\|_{s-})}\mu_N(du)\leq \int_{L^2_a}\frac{\|u\|^2}{a^2}e^{3\xi^{-1}(\|u\|_{s-})}\mu_N(du)\leq \frac{A_0^N}{2a^2}.\label{est_muN_e}
\end{align}
Now using \eqref{est_muN_E}, \eqref{DefmathcalE}, we obtain
\begin{align}
\int_{L^2}\left[\|u\|_s^2+\left(\|u\|_1^2+\|u\|_{L^{p+1}}^{p+1}\right)e^{3\xi^{-1}(\|u\|_{s-})}\right]\mu_N(du)\leq C_1+K:= C,\label{ReE}
\end{align}
where $C$ does not depend on $N.$

\begin{prop}
There are a subsequence $(\mu_{\psi(N)})\subset (\mu_N)$ and a measure $\mu$ on $L^2$ such that
\begin{align}
\lim_{N\to\infty}\mu_{\psi(N)}=\mu, \quad weakly\ on\ H^{r},\ \forall r<s. \label{Conv_mes_muNtomu} 
\end{align}
Moreover, we have the estimates
\begin{align}
\int_{L^2}\mathcal{M}(u)\mu(du) &= \frac{A_{0}}{2}, \label{est-phi-M}\\
\int_{L^2}\left[\|u\|_s^2+\left(\|u\|_1^2+\|u\|_{L^{p+1}}^{p+1}\right)e^{3\xi^{-1}(\|u\|_{s-})}\right]\mu(du) &\leq C, \label{est-phi-E}
\end{align} 
where $C$ does not depend on $t$.
\end{prop}
\begin{proof}
The independence in $N$ of the constance $C$ in \eqref{est_muN_E} ensures the tightness of the sequence $(\mu_{N})$ on $H^r,\ r<s$, thanks to the Prokhorov theorem. We obtain the first statement of the proposition.\\
The estimate \eqref{est-phi-E} follows from \eqref{ReE} and the lower semicontinuity of $\mathcal{E}(u)$. Now let us prove \eqref{est-phi-M}: let $\chi_R$ be a cut-off function on $\R$ having the value $1$ on $[0,1]$ and the value $0$ on $[2,\infty),$ we write
\begin{align*}
\frac{A_{0,N}}{2}-\int_{L^2}(1-\chi_R(\|u\|^2))\mathcal{M}(u)\mu_{N}(du)\leq \int_{L^2}\chi_R(\|u\|^2)\mathcal{M}(u)\mu_{N}(du)\leq \frac{A_{0,N}}{2}.
\end{align*}
Now, we obtain, with the use of \eqref{est_muN_MR},
\begin{align*}
\frac{A_{0,N}}{2}-C_2R^{-1}\leq \int_{L^2}\chi_R(\|u\|^2)\mathcal{M}(u)\mu_{N}(du)\leq \frac{A_{0,N}}{2}.
\end{align*}
It remains to pass to the limits $N\to\infty$, then $R\to\infty$ to arrive at the claim. 
\end{proof}
Recall that $E^N_a=\{u\in E_N^a\ |\ \|u\|\geq a\}$, for fixed small enough $a>0$.
\begin{prop}\label{prop_control}
Let $s-\frac{1}{2}<r\leq s-$,  $N\geq 0$ and $a>0$. There $C=C(a)>0$, such that for any $i\in\N^*$, there is a set $\Sigma^i_{N,r}$ verifying
\begin{align}
\mu_N(E_N^a\backslash\Sigma_{N,r}^i )\leq Ce^{-2i},\label{ConstrSigma}
\end{align}
and having the property: For all $u_0\in\Sigma^i_{N,r}$, we have
\begin{align}
\|\phi_N^tu_0\|_r\leq 2\xi(1+i+\ln(1+|t|)), \quad \forall t\in\R.\label{estimeePolynomT}
\end{align}
\end{prop}
\begin{proof}
Without loss of generality, let us work with non-negative times. Define, for $j\geq 1,$  the set  
\begin{align}
B^{i,j}_{N,r}=\left\{u\in E_N^a|\ \ \|u\|_r\leq \xi(i+j)\right\},\label{Bij}
\end{align}
Let $T\sim [\xi(i+j)]^{1-p},$ this is smaller than the time existence defined in Proposition $\ref{propLWP}.$ Then, according to the same proposition, we know  that for $t\in [0,T],$ 
\begin{align}
\phi_N^tB^{i,j}_{N,r}\subset \{u\in E_N^a |\ \ \|u\|_r\leq 2\xi(i+j)\}.\label{Injection_Bij}
\end{align}
Define the set
\begin{align}
\Sigma^{i,j}_{N,r}=\bigcap_{k=0}^{\left[\frac{e^j}{T}\right]}\phi_N^{-kT}(B^{i,j}_{N,r}).\label{Sigij}
\end{align}
Using the invariance of $\mu_N$ under $\phi_N^t,$ we have
\begin{align*}
\mu_N(E_N^a\backslash \Sigma^{i,j}_{N,r})=\mu_N\left(\bigcup_{k=0}^{\left[\frac{e^j}{T}\right]} E_N^a\backslash (B^{i,j}_{N,r})\right)\leq \left(\left[\frac{e^j}{T}\right]+1\right)\mu_N(E_N^a\backslash B^{i,j}_{N,r}).
\end{align*}
Now since $r\leq s-$, we have from \eqref{est_muN_e} that $\E e^{3\xi^{-1}(\|u\|_r)}\leq C,$ where $C_1=C_1(a)>0$ is a constant independent of $r$ and $N$. One has, with the use of the Chebyshev inequality,
\begin{align}
\mu_N(E_N^a\backslash \Sigma^{i,j}_{N,r}) \leq C_1\left[\frac{e^j}{T}\right]e^{-3\xi^{-1}(\xi(i+j))} &\lleq e^j[\xi(i+j)]^{p-1}e^{-3(i+j)}\nonumber\\
&\lleq  e^{-2i}e^{-j}.
\end{align}
Now let us define the needed set as
\begin{align}
\Sigma^i_{N,r}=\bigcap_{j\geq 1}\Sigma^{i,j}_{N,r}.\label{Sigi}
\end{align}
We verify easily \eqref{ConstrSigma} using the fact that the series $\sum_{j\geq 1}e^{-j}$ converges.\\
Next, let us observe that for $u_0\in\Sigma^{i,j}_{N,r},$ we have
\begin{align}
\|\phi_{N}^tu_0\|_r\leq 2\xi(i+j)\quad \forall\ t\leq e^j.\label{IneqflowN2j}
\end{align}
Indeed, for $t\leq e^j$, we can write $t=kT+\tau$, where $k$ is an integer in $[0,\frac{e^j}{T}]$ and $\tau \in [0,T]$. Also, by definition of $\Sigma_{N,r}^{i,j}$, we have that $u_0$ can be written as $\phi^{-kT}w$ for any fixed integer $k \in [0,\frac{e^j}{T}]$ and a corresponding $w \in B_{N,r}^{i,j}$.  We then have
\begin{align*}
\phi^t u_0=\phi^\tau \phi^{kT}u_0=\phi^\tau w.
\end{align*}
Now, using \eqref{Injection_Bij}, we obtain \eqref{IneqflowN2j}.

Let $t\geq 0,$ there is $j\geq 1$ and  such that $e^{j-1}\leq 1+t\leq e^{j}$, therefore 
\begin{align*}
j-1\leq \ln (1+t),
\end{align*}
then
\begin{align*}
j\leq 1+\ln (1+t).
\end{align*}
And then,
\begin{align*}
\|\phi_N^t u_0\|_r\leq 2\xi(i+j)\leq 2\xi(1+i+\ln(1+t)),
\end{align*}
then we arrive at the estimate \eqref{estimeePolynomT}.
\end{proof}
\begin{prop}\label{Propexists1}
For any $s-\frac{1}{2}<r\leq s-,$ any $s-\frac{1}{2}<r_1<r$, for every $t\in\R$, there is $i_1\in\N^*$ such that for any $i\in\N^*$, if $u_0\in \Sigma_{N,r}^i,$ then we have $\phi_N^t(u_0)\in\Sigma_{N,r_1}^{i+i_1}.$
\end{prop}
\begin{proof}
Fix $r\in(s-\frac{1}{2},s-].$ Without loss of generality, assume $t>0.$ Let $u_0\in \Sigma_{N,r}^i$, then for any $j\geq 1$, we have
\begin{align*}
\|\phi_N^{t_1}u_0\|_r\leq 2\xi(i+j),\quad t_1\leq e^{j}.
\end{align*}
Let $i_1:=i_1(t)$ be such that for every $j\geq 1,\ e^j+t\leq e^{j+i_1}$.  We then have
\begin{align*}
\|\phi_N^{t_1+t}u_0\|_r\leq 2\xi(i+j+i_1),\quad t_1\leq e^{j}. 
\end{align*}
Now, thanks to \eqref{estimeePolynomT}, we have, for every $u_0\in\Sigma_{N,r}^i,$
\begin{align*}
\|u_0\|\leq\|u_0\|_r\leq 2\xi(1+i),
\end{align*}
therefore, since the $L^2-$norm is preserved, we have, for every $u_0\in \Sigma_{N,r}^i$, that
\begin{align*}
\|\phi_N^{t_1+t} u_0\|\leq 2\xi(1+i).
\end{align*}
Hence for every $r_1\in (s-\frac{1}{2},s-)$, we use an interpolation to see that there is $\theta\in (0,1)$ such that
\begin{align*}
\|\phi_N^{t+t_1}u_0\|_{r_1}&\leq \|\phi_N^{t+t_1}u_0\|^{1-\theta}\|\phi_N^{t+t_1}u_0\|_{r}^{\theta}\leq 2^{1-\theta}2^{\theta}(\xi(1+i))^{1-\theta}(\xi(i+j+i_1))^{\theta}\\
&\leq \xi(i+j+i_1).
\end{align*}
the last inequality above follows the fact that $j\geq 1$ and $\xi$ is an increasing function, therefore the inequality holds for $i_1(t)$ large enough. Thus we obtain that $\phi^{t+t_1}_N(u_0)$ belongs to $B^{i_1,j}_{N,r_1}$ for the constructed $i_1(t)$ and $r_1$, for all $t_1\leq e^j$, for all $j\geq 1$. The proof is finished.
\end{proof}

Let us introduce the restriction measures (or conditional probabilities)
\begin{align*}
\mu_{N,i,r}(\Gamma)=\mu_N(\Gamma\ |\ \Sigma_{N,r}^i)=
\frac{\mu_{N}(\Gamma\cap\Sigma_{N,r}^i)}{\mu_N(\Sigma_{N,r}^i)},\quad \Gamma\in  \text{Bor}(L^2).
\end{align*}
We do not claim any invariance of these measures under the corresponding dynamics.
\begin{prop}
For any $i\in \N^*$ any $r<s,$ the sequence $(\mu_{N,i,r})_{N\geq 1}$ is tight on $H^r,$ $r<s$. In particular, there is a subsequence that we denote by $(\mu_{N,i,r})$ and that converges weakly to a measure $\mu_{i,r}$ on $H^r,\ r<s.$
\end{prop}
\begin{proof}
We see, using \eqref{ReE}, that 
\begin{align*}
\E_{\mu_{N,i,r}}\|u\|_s^2\leq \frac{\E_{\mu_N}\|u\|_s^2}{\mu_N(\Sigma_{N,r}^i)}\leq\frac{C}{1-Ce^{-2i}}.
\end{align*}
This gives the claimed tightness by using the Chebyshev theorem. The compactness follows from the Prokhorov theorem.
\end{proof}
Now, by invoking the Skorokhod representation theorem (see Theorem $11.7.2$ in \cite{dudley}), we obtain a probability space still denoted $(\Omega,\P)$ on  which are defined random variables $u_{N,i,r}$ and $u_{i,r}$ satisfying the following
\begin{enumerate}
\item $u_{i,r}$ is distributed by $\mu_{i,r}$, and for every $N$, $u_{N,i,r}$ is distributed by $\mu_{N,i,r}$;
\item $u_{N,i,r}$ converges to $u$ almost surely in $H^r.$ 
\end{enumerate}
Let us introduce the sets
\begin{align*}
\Sigma_r^i=\{u\in H^r|\ \ \exists \ N_k\to\infty\ as \ k\to\infty,\ \exists (u_{N_k}),\ u_{N_k}\to u\ as\ k\to\infty\ \ and\ \ u_{N_k}\in \Sigma_{N_k,r}^i\}.
\end{align*}
\begin{rmq}
\begin{enumerate}
\item For any fixed $N$, $\Sigma_{N,r}^i$ is obviously included in $\Sigma_r^i$, for instance we can take constant sequences of  $\Sigma_{N,r}^i$. 
\item We have that $(\Sigma^i_r)_i$ is non-decreasing. Indeed
$(B^{i,j}_{N,r})_i$ is non-decreasing (see \eqref{Bij}). Then, by definition, so does $(\Sigma^{i,j}_{N,r})_i$ (see \eqref{Sigij}) and then $(\Sigma^i_{N,r})_i$ (see \eqref{Sigi}). And it is clear that this property is preserved by the definition of $\Sigma^i_r.$
\end{enumerate}
\end{rmq}
Let us set
$\Sigma_r=\bigcup_{i\geq 1}\overline{\Sigma_{r}^i}$.
\begin{prop}\label{PropComparaisonMeasures}
The following holds
\begin{enumerate}
\item\label{inclusion1} The support of $\mu_{i,r}$ is contained in $\Sigma_r^i$, up to a set of $\mu_{i,r}-$measure $0.$ Hence
$\mu_{i,r}(\Sigma_r^i)=1.$
\item\label{inclusion2} We have that
\begin{align}
\mu(\Sigma_r)=1.
\end{align}
\item\label{inclusionInequ} For any $f\in C_b(H^r)$ bounded by $1$, we have the inequalities
\begin{align}
\mu(f) &\leq \mu_{i,r}(f)+Ce^{-2i},\label{inclusionIneq1}\\
\mu_{i,r}(f) &\leq \frac{1}{1-Ce^{-2i}}\mu(f),\label{inclusionIneq2}
\end{align}
where $v(f):=\int_{L^2}f(u)\nu(du).$ In particular
\begin{align*}
\lim_{i\to\infty}\mu_{i,r}=\mu \quad\text{weakly on $H^r$, $r<s.$}
\end{align*}
\end{enumerate}
\end{prop}
\begin{proof}
Using the Skorokhod representation theorem, we have that the support of $\mu_{i,r}$ contains essentially the almost sure limits of a sequence of random variables whose elements are distributed by the measures $\mu_{N,i,r}$, respectively. Now, by definition of $\mu_{N,i,r}$, these Skorokhod's random variables distribute in  $\Sigma_{N,r}^i$ respectively. Hence, we get the inclusion in \ref{inclusion1}.\\
Next, using the Portmanteau theorem, the inclusion $\Sigma^i_{N,r}\subset\Sigma^i_r$ and then the definition of $\Sigma^i_{N,r}$, we have
\begin{align*}
\mu(\overline{\Sigma_{r}^i})\geq\lim_{N\to\infty} \mu_{N}(\overline{\Sigma_r^i})\geq\lim_{N\to\infty} \mu_{N}(\Sigma_r^i)\geq\lim_{N\to\infty} \mu_{N}(\Sigma_{N,r}^i)\geq 1-Ce^{-2i}.
\end{align*}
Since $(\Sigma^i_r)_{i\geq 1}$ is non-decreasing, then so does $(\overline{\Sigma^i_r})_{i\geq 1}$, therefore we obtain
\begin{align*}
\mu(\Sigma_r)=\mu\left(\cup_{i\geq 1}\overline{\Sigma^i_r}\right)=\lim_{i\to\infty}\mu(\overline{\Sigma^i_r})\geq 1.
\end{align*}
Since $\mu$ is a probability measure, we get
\begin{align*}
\mu(\Sigma_r)=1.
\end{align*}
Next, let us prove the inequalities in point \ref{inclusionInequ}:
\begin{align*}
\int_{L^2}f(u)\mu_{i,r}(du)=\lim_{N\to\infty}\int_{L^2}f(u)\mu_{N,i,r}(du) \leq \frac{1}{1-Ce^{-2 i}}\lim_{N\to\infty}\int_{L^2}f(u)\mu_{N}(du)=\frac{1}{1-Ce^{-2 i}}\int_{L^2}f(u)\mu(du),
\end{align*}
that is \eqref{inclusionIneq2}. Also, using the fact that $\mu_N(\Sigma_{N,r}^i)\leq 1,$ we have
\begin{align*}
\int_{L^2}f(u)\mu_N(du) &=\int_{\Sigma_{N,r}^i}f(u)\mu_N(du)+\int_{L^2\backslash \Sigma_{N,r}^i}f(u)\mu_N(du)\\
&\leq \int_{L^2}f(u)\mu_{N,i,r}(du)+\mu_N(L^2\backslash \Sigma_{N,r}^i)\leq \int_{L^2}f(u)\mu_{N,i,r}(du)+Ce^{-2 i}.
\end{align*}
After passing to the limit $N\to\infty$, we obtain the inequality \eqref{inclusionIneq1}.
\end{proof}
Now, we state the well-posedness result.
\begin{prop}\label{PropGWP}
Let $r\leq s-$. For any $u_0\in\Sigma_r\cap \text{Supp}(\mu)$, there is a unique global in time solution to \eqref{NLS7totalequ}. Therefore we obtain a global flow $\phi^t$ defined on $\Sigma_s.$\\
 For any  $T_0>0,$ there is $C(T_0)>0$ such that for any $u,v\in \Sigma_r$
\begin{align}
\sup_{t\in [-T_0,T_0]}\|\phi^t(u)\|_{r} &\leq C(T_0), \label{NL7-estesptime}\\
\sup_{t\in [-T_0,T_0]}\|\phi^t(u)-\phi^t(v)\|_{r} &\leq C(T_0)\|u-v\|_r,\label{continuity}\\
\|\phi^t(u)\|_{s-} &\leq 2\xi(1+i+\ln(1+|t|))\label{ControlModerate}.
\end{align}
\end{prop}
\begin{proof}
Let us fix an arbitrary $T_0>0.$ Recall that $\mu(\Sigma_r\cap\text{Supp}(\mu))=1$ (Proposition \ref{PropComparaisonMeasures}). We wish to show that for $u_0\in \Sigma_r\cap\text{Supp}(\mu)$, the solution $\phi^t u_0$ constructed in Proposition $\ref{propLWP}$ exists in fact on $[-T_0,T_0].$ Then assume that $T_0$ is greater than the time of Proposition $\ref{propLWP}$. Remark also that, from the bound \eqref{est-phi-E}, $\text{Supp}(\mu)\subset H^s.$ Then in particular, $u_0\in H^s.$\\
Now, by the construction of $\Sigma_r$, any $u_0$ in $\Sigma_r$ belongs to $\overline{\Sigma_{r}^i}$ for some $i.$ Let us consider the two cases (not necessary disjoint):
\begin{itemize}
\item  $u_0\in \Sigma_{r}^i$
\item $u_0\in \partial \Sigma_{r}^i$
\end{itemize} 

In the first case, there is  a sequence $(u_{0,N})_N$ such that $u_{0,N}\in \Sigma_{N,r}^i$. Using the estimate \eqref{estimeePolynomT}, we have that
\begin{align*}
\|\phi_N^tu_{0,N}\|_{r}\leq 2\xi(1+i+\ln(1+|t|)),\quad t\in\R.
\end{align*}
Therefore we have the bound
\begin{align*}
\|\phi_N^tu_{0,N}\|_{r}\leq 2\xi(1+i+\ln(1+|T_0|)),\quad |t|\leq T_0.
\end{align*}
And, at $t=0$, we see that
\begin{align*}
\|u_{0,N}\|_{r}\leq 2\xi(1+i),
\end{align*}
hence, by passing to the limit $N\to\infty$,
\begin{align}
\|u_{0}\|_{r}\leq 2\xi(1+i).\label{ineq_u_0_2e1i}
\end{align}
Let us remark for $u_0\in \partial\Sigma_{r}^i$, there is a sequence $(u_0^k)_k\in \Sigma_{r}^i$ that converges to $u_0$ in $H^r.$ We see easily that \eqref{ineq_u_0_2e1i} holds also on $\partial\Sigma_{r}^i$ and then on $\overline{\Sigma_{r}^i}.$

Set $\Lambda=2\xi(1+i+\ln(1+|T_0|)),$ and $R=\Lambda+1$. From Proposition \ref{propLWP}, we have a uniform existence time associated to the ball $B_{R}$ is greater than $T\sim R^{1-p}.$ Let $u_0\in \Sigma_{r}^i$ and $|t|\leq T,$ we have that 
\begin{align*}
\phi^t(u_0)-\phi_N^t(u_{N,0}) =S(t)(u_0-u_{N,0}) &-\i\int_0^tS(t-\tau)\left(P_N(|\phi^{\tau}(u_0)|^{p-1}\phi^{\tau}(u_0)-|\phi_N^{\tau}(u_{N,0})|^{p-1}\phi_N^{\tau}(u_{N,0}))\right)d\tau\\
&-\i\int_0^tS(t-\tau)\left((1-P_N)|\phi^{\tau}(u_0)|^{p-1}\phi^{\tau}(u_0)\right)d\tau.
\end{align*}
Therefore, using in particular the fact that $u_0$ belongs to $H^s$ (implying that $\phi^t u_0\in H^s$ for $|t|\leq T$) to treat the last term in the RHS, we have 
\begin{align*}
\|\phi^t(u_0)-\phi_N^t(u_{N,0})\|_r\leq \|u_0-u_{N,0}\|_r+C(r)\int_0^{t}\|\phi^{\tau}(u_0)-\phi_N^{\tau}(u_{N,0})\|_rd\tau+C\lambda_{N}^{\frac{r-s}{2}}.
\end{align*}
Using the Gronwall lemma, and letting $N$ go to $\infty,$ we that
\begin{align*}
\|\phi^t(u_0)-\phi_N^t(u_{N,0})\|_r\to 0,\quad |t|\leq T.
\end{align*}
Now, by the triangle inequality
\begin{align*}
\|\phi^t(u_0)\|_r\leq \|\phi^t(u_0)-\phi_N^t(u_{0,N})\|_r+\|\phi_N^tu_{0,N}\|_r\leq \|\phi^t(u_0)-\phi_N^t(u_{0,N})\|_r+\Lambda, 
\end{align*}
passing to the limit on $N,$ we obtain
\begin{align}
\|\phi^t(u_0)\|_r\leq \Lambda\quad |t|\leq T.\label{localNLSbound}
\end{align}
Then $\phi^T(u_0)$ still belongs to the ball $B_R,$ and we can iterate the procedure. Repeating the argument above, we have 
\begin{align*}
\|\phi^t(u_0)-\phi_N^t(u_{N,0})\|_r\leq \|\phi^T(u_0)-\phi_N^T(u_{N,0})\|_r+C\int_T^{t}\|\phi^{\tau}(u_0)-\phi_N^{\tau}(u_{N,0})\|_rd\tau+C\lambda_{N}^{\frac{r-s}{2}} \quad T\leq |t|\leq 2T.
\end{align*}
Again, we obtain that for $T\leq |t|\leq 2T,$ $\|\phi^{t}(u_0)-\phi_{N}^{t}(u_{0,N})\|_r\to 0$ as $N\to 0$, leading to the estimate $\|\phi^{t}(u_0)\|_r\leq \Lambda,\ \ T\leq |t|\leq 2T$, as above. We see that after the $n'th$ step, $\phi^{nT}(u_0)$ remains in the ball $B_{\Lambda}$, allowing the next iteration. Then we arrive at the claim after iterating a sufficient number of times (recall that $\|\phi_N^t(u_{0,N})\|_r$ remains bounded by $\Lambda$ on $[-T_0,T_0]$.)\\
The bound \eqref{NL7-estesptime} for $u_0\in \Sigma_{r}^i$ follows from the iteration of \eqref{localNLSbound}.\\
Now let $u_0\in\partial\Sigma_{r}^i$, take a $(u_0^k)_k\subset\Sigma_{r}^i$  converging to $u_0$. Recall that the bound \eqref{ineq_u_0_2e1i} holds for both $u_0^k$, for all $k,$ and $u_0$. In particular these elements belong to the ball $B_R$ where $R=\Lambda+1$, the same as above. Denote again the time existence of this ball by $T$. We have by continuity that
\begin{align*}
\lim_{k\to\infty}\|\phi^t(u_0)-\phi^t(u_0^k)\|_r=0.
\end{align*}
Combining this convergence with the triangle inequality, we have
\begin{align*}
\|\phi^t(u_0)\|_r\leq \Lambda\quad \forall t\leq T.
\end{align*}
This allows to iterate the procedure as above. We arrive at global existence for data in $\partial\Sigma_{r}^i$ and completed the globalization on $\overline{\Sigma_{r}^i}.$ Also \eqref{NL7-estesptime} is established on $\overline{\Sigma_{r}^i}.$

Now, using the Duhamel formula, it is not difficult to see that
\begin{align*}
\|\phi^tu-\phi^tv\|_r &=\|u-v\|_r +C\int_0^t\left(\|\phi^{\tau}v\|_{L^\infty}^{p-1}+\|\phi^{\tau}u\|_{L^\infty}^{p-1}\right)\|\phi^\tau u-\phi^\tau v\|_rd\tau\\
&\leq \|u-v\|_r +2C\Lambda^{p-1}\int_0^t\|\phi^\tau u-\phi^\tau v\|_rd\tau.
\end{align*}
We use the Gronwall lemma and take the sup over $[-T_0,T_0]$ to obtain \eqref{continuity}. The inequality \eqref{ControlModerate} follows from \eqref{estimeePolynomT}.
\end{proof}
\begin{rmq}
From the proof above, we have that for any $i\geq 1,$ any $u_0\in \Sigma_r^i,$ any $t\in \R$,
\begin{equation}\label{Convergence_phit-phitN}
\lim_{N\to\infty}\|\phi^tu_0-\phi_N^{t}u_{0,N}\|_r=0,
\end{equation}
where $(u_{0,N})$ is a sequence in $\Sigma_{N,r}^i$ that converges to $u_0$ in $H^r.$
\end{rmq}

Consider an increasing sequence  $l=(l_n)_{n\in\N}$ such that $l_0=s-\frac{1}{2}$ and $\lim_{n\to\infty}l_n=s-.$ Set
\begin{align*}
\Sigma=\bigcap_{r\in l}\Sigma_r.
\end{align*} 

We have the following result.
\begin{prop}
The set $\Sigma$ is of full $\mu-$measure. Moreover, the flow $\phi^t$ constructed in Proposition \ref{PropGWP} satisfies $\phi^t\Sigma=\Sigma,$ for any $t\in\R.$
\end{prop}
\begin{proof}
Since any $\Sigma_r$ is of full $\mu-$measure and the intersection is countable, we obtain the first statement.\\
To prove the second statement, let us take $u_0\in \Sigma$, then $u_0$ belong to each $\Sigma_r$, $r\in l.$ \\
First, consider $u_0\in \Sigma^i_r$. Therefore $u_0$ is the limit of a sequence $(u_{0,N})$ such that $u_{0,N}\in \Sigma_{N,r}^i$ for every $N$. Now from the Proposition \ref{Propexists1}, there is $i_1:=i_1(t)$ such that $\phi_N^t(u_{0,N})\in \Sigma_{N,r_1}^{i+i_1}$.
Using the convergence \eqref{Convergence_phit-phitN}, we see that $\phi^t(u_0)\in \Sigma_{r_1}^{i+i_1}$. Now if $u_0\in \partial\Sigma_r^i$, there is $(u_0^k)_k\subset\Sigma^i_r$ that converges to $u_0$ in $H^r$.  Since we showed that $\phi^t\Sigma^i_r\subset\Sigma^{i+i_1}_{r_1}$ and $\phi^t(\cdot) $ is continuous, we see that $\phi^t(u_0)=\lim_{k}\phi^t(u_0^k)\in \overline{\Sigma^{i+i_1}_{r_1}}$. We conclude that $\phi^t\overline{\Sigma^i_r}\subset\overline{\Sigma_{r_1}^{i+i_1}}\subset \Sigma_{r_1}$. It follows that $\phi^t\Sigma\subset\Sigma.$

 Now, let $u$ be in $\Sigma$, since $\phi^t$ is well-defined on $\Sigma$ we can set  $u_0=\phi^{-t}u$ , we then have $u=\phi^tu_0$ and hence $\Sigma\subset\phi^{t}\Sigma.$ That finishes the proof.
\end{proof}
\section{Invariance of the measure}
\label{ASection6InvarMeas}
\begin{thm}
The measure $\mu$ is invariant under $\phi^t.$
\end{thm}
\begin{proof}
The measure $\mu$ is a Borel probability defined on a Polish space.  The Ulam's theorem (see Theorem $7.1.4$ in \cite{dudley}) states that such a measure is regular: for any $S\in \text{Bor}(H^s)$
\begin{align*}
\mu(S)=\sup\{\mu(K),\ K\subset S\ compact\}.
\end{align*}
Therefore it suffices to prove invariance for compact sets. Indeed, we then obtain, for any $t$,
\begin{align}
\mu(\phi^{-t}S)&=\sup\{\mu(K),\ K\subset \phi^{-t}S\ compact\}=\sup\{\mu(\phi^{t}K),\ K\subset \phi^{-t}S\ compact\}\\
&=\sup\{\mu(\phi^{t}K),\ \phi^{t}K\subset S,\ K\  compact\}\leq \sup\{\mu(C),\ C\subset S\ compact\} =\mu(S),
\end{align}
where we used the fact that $\phi^t$ is continuous in space, therefore it transforms compact sets into compact sets.\\
Using the inequality above, we also have for any $t$ that
\begin{align*}
\mu(S)=\mu(\phi^{-t}\phi^t S)\leq \mu(\phi^tS).
\end{align*}
since $t$ is arbitrary, we then obtain the invariance.\\
Now we claim that it also suffices to show the invariance only on a fixed interval $[-\tau,\tau],$ where $\tau>0$ can be as small as we want. Indeed for $\tau\leq t\leq 2\tau$, one has $\mu(\phi^{-t}K)=\mu(\phi^{-\tau}\phi^{t-\tau}K)=\mu(\phi^{t-\tau}K)=\mu(K)$ (using that $0\leq t-\tau\leq\tau$), and for greater values of $t$ we can iterate. A same argument works for negative values of $t.$\\
Our proof is then reduced to showing invariance for compact sets on a small time interval. Therefore, it suffices to show it on the balls of $H^s.$
Here is the idea of the proof:
$$
\hspace{10mm}
\xymatrix{
  \Phi_{N}^{*t}\mu_{k} \ar@{=}[r]^{(I)} \ar[d]^{(III)} & \mu_{k} \ar[d]^{(II)} \\
    \Phi^{*t}\mu \ar@{=}[r]^{(IV)} & \mu
  }
$$
The equality $(I)$ is the invariance of $\mu_N$ under $\Phi_{N}^t$, and $(II)$ is the weak convergence $\mu_N\to\mu$. Then $(IV)$ is proved once $(III)$ is verified.

Let $f \in C_b(H^s)$,  supported on a ball $B_R(H^s).$ Assume that $f$ is  Lipschitz in the topology of $H^r,\ r<s$. Let $\tau$ be the associated time existence provided by Proposition \ref{propLWP}. Then for $t<\tau,$ we have
\begin{align*}
(\Phi^{t*}_{N}\mu_N,f)-(\Phi^{*t}\mu,f)&=(\mu_N,\Phi_{N}^tf)-(\mu,\Phi^tf)\\
&=(\mu_N,\Phi_{N}^tf-\Phi^tf)-(\mu-\mu_N,\Phi^tf)\\
&=A-B.
\end{align*}
By the continuity property of $\phi^t$, we have that $\Phi^tf\in C_b(H^s).$ Then by weak convergence of $\mu_N$ to $\mu$ on $H^r$, we have that $B\to 0$ as $N\to \infty$.\\

Now using the Lipschitz property of $f$, we have, with the use of Lemma \ref{LemUnifConv},
\begin{align*}
|A|\leq C_f\sup_{u\in B_R(H^s)}\|\phi_N^t(u)-\phi^t(u)\|_r\mu_N(B_R(H^r))\leq C_f\sup_{u\in B_R(H^s)}\|\phi_N^t(u)-\phi^t(u)\|_r\to 0,\ \ as\ N\to\infty.
\end{align*}
We obtain the claim.
\end{proof}
\section{Almost sure GWP on $H^s(\T^3)$ and remark on the size of the data}
\label{ASection7GWPH2Size}
We have shown the global well-posedness on the support of $\mu$ viewed as a subset of $\cap_{\sigma<s-}H^\sigma$ (Proposition \ref{PropGWP}). But the estimate \eqref{est-phi-E} (in particular the control on $\|u\|_s$) shows us  that $\mu$ is in fact concentrated on $H^s.$  As a consequence, we give here the argument that the global well-posedness holds with respect to the topology of $H^s.$ This fact relies on the propagation of regularity principle, very well known in the context of dispersive equations. Afterwards, we give an argument showing that large data are concerned by our result.\\ 
From Subsection \ref{subsectglobalizationcondition}, we have the statement that if the quantity $\int_0^t\|\phi^tu_0\|_{L^\infty}^{p-1}d\tau$ remains finite for all times, then the solution issued from $u_0\in H^s$ is global in $H^r.$ Now let $u_0$ belong to the support of $\mu,$ thanks to Proposition \ref{PropGWP}, the solution of \eqref{Equ_NLS7} issued to $u_0$ is global and belongs to $C_tH^r$ for any $r\in (s-\frac{1}{2},s)$, in particular the quantity $\int_0^t\|\phi^tu_0\|_{L^\infty}^{p-1}d\tau$ remains finite for all $t.$ By this way, we see that the local solutions on $H^s$ stated in Proposition \ref{propLWP} are global on the support on $\mu$ viewed as a subset of $H^s.$ The invoked control allows also uniqueness and continuity with respect to the initial datum by following usual estimation procedures.

Now let us turn our attention to the size of the data. We remark that the ensemble constructed in this work does not concern only small data. In fact, by an \textit{scaling of the measure}, we have that for any $\Lambda>0,$ there  is a non-degenerate measure $\mu^\Lambda$ concentrated on $H^s$ such that
\begin{align}
\E_{\mu^\Lambda}\mathcal{M}(u)=\Lambda,\label{scaling estimate}
\end{align}
and we have global wellposedness on the support of $\mu^{\Lambda}.$ 
 To see the construction of such a measure, it suffices to change the numbers $(a_m)$ entering the definition of the noise in \eqref{formula_BM} into $(\frac{a_m\sqrt{\Lambda}}{\sqrt{A_0}})$. Therefore, the number $A_0$ is changed into $\Lambda,$ the numbers $A_{0,N}$ into $\Lambda_N:=\frac{A_{0,N}}{A_0}\Lambda$ that converge clearly to $\Lambda.$ Also, all the analysis done here remains unchanged (because the scaling in $\alpha$ between the fluctuation and the dissipation in \eqref{eqN} is not affected: we still keep $\alpha$ as the size of the dissipation for a fluctuation of intensity $\sqrt{\al}$). Therefore the following statement is a consequence of the results that have been establish so far:
 \begin{thm}\label{THMCUM}
 Let $\Lambda >0$, there is a measure $\mu^\Lambda$ concentrated on $H^s$ and having the following properties
 \begin{enumerate}
 \item The  NLS equation \eqref{Equ_NLS7} is globally well-posed on the support of $\mu^\Lambda;$
 \item The identity \eqref{scaling estimate} holds true;
 \item The measure $\mu^\Lambda$ is invariant under the flow $\phi^t$ of \eqref{Equ_NLS7} defined on its support $S^\Lambda$.
 \end{enumerate}
 \end{thm}
Recall that $\mathcal{M}(u)=\|u\|_{s-1}^2+e^{\xi(\|u\|_{s-})}\|u\|^2\leq \left(1+e^{\xi(\|u\|_{s-})}\right)\|u\|_{s-}^2$. Therefore, the estimate \eqref{scaling estimate} provides data on the support  $S^\Lambda$of $\mu^\Lambda$ whose $H^{s-}-$sizes are larger than $C(\Lambda),$ where $C(\Lambda)\to\infty$ as $\Lambda\to\infty$. We see from \eqref{scaling estimate} that the set of such data is of positive $\mu^\Lambda-$measure.\\
Furthermore, we can define a cumulative probability measure
\begin{align*}
\mu^*=\sum_{n=1}^{\infty}\frac{\mu^{n}}{2^n},
\end{align*}
where we have taken $\Lambda=n$,  $n\in\N^*.$ The support of $\mu^*$ is the set
$$S^*=\bigcup_{n\in\N^*}S^n.$$
It follows from Theorem \ref{THMCUM} that a global flow for \eqref{Equ_NLS7} that we write again $\phi^t$ is defined  on $S^*.$ \\
Since for any $n,$ $\E_{\mu^n}\mathcal{M}(u)=n$, we have that for any $n,$ there is a set of positive $\mu_n-$measure containing initial data whose sizes are bigger than $C(n)$, where $C(n)$ goes to infinity with $n$. Hence, we obtain the following statement:
\begin{align*}
\forall n>0,\ there\ is \ a\ set\ W_n\ such\ that\ \mu^*(W_n)>0,\ and \ any \ u_0\in W_n\ satisfies\ \|u_0\|_s\geq n.
\end{align*}
 Moreover since $\phi^t_*\mu^n=\mu^n$ for any $n$, we see that $\mu^*$ is invariant under the flow $\phi^t$.  This finishes the discussion of this section.
\section{Density for the distributions of the conservation laws.}
\label{ASection8Qualprop}
Let $\mu_{\al,N}$ be an stationary measure of \eqref{eqN} and $\mu$ the invariant measure for \eqref{Equ_NLS7} that has been constructed in the previous sections. The quantity $E_*\mu=\mu(E\in \cdot)$ is the law of $E(u)$, where $u$ is distributed as $\mu$. The similar notation is used for $M$ and for the measures $\mu_{\al,N}.$ 
\begin{thm}\label{thm_density}
Suppose  $a_m$ is non-zero for any $m\geq 0.$
Then, the measures  $E_*\mu$ and $M_*\mu$  are absolutely continuous with respect to the Lebesgue measure on $\R.$
\end{thm}
Before presenting the proof of the theorem, we establish some results concerning a quite  general context. 
Consider a general equation
\begin{align*}
du=f(u)dt+d\zeta,
\end{align*}
where $\zeta$ is a Brownian motion in some separable Hilbert space $X$, given by
$$\zeta(t,x) =\sum_{|m|\geq 0}a_me_m(x)\beta_m(t),$$
where the parameters entering the sum are similar to \eqref{formula_BM}.
 Suppose that the equation admits an stationary measure $\nu$ concentrated on $X$, the corresponding solution is denoted by $u$. For a functional $F:X\to\R,$ we denote by $F_*\nu$ the distribution of $F(u)$, that is $F_*\nu(\cdot)=\nu(F^{-1}(\cdot)).$
\begin{thm}\label{thm_density_general}
Let $F$ be in $C^2(X,\R)$ satisfying the It\^o change of variable
$$dF(u)=\left(F'(u;f(u))+\frac{1}{2}\sum_{|m|\geq 0}|a_m|^2F''(u;e_m,e_m)\right)dt+\sum_{|m|\geq 0}a_mF'(u;e_m)d\beta_m.$$
 Let $O\subset X$ be an open set and $c$ and $C$ be two positive constants such that
\begin{align}
Q(v):=\sum_{|m|\geq 0} |a_m|^2|F'(v,e_m)|^2&\geq c\quad for\ \nu-almost\  all\ \ v\ in\ O, \label{condition_sur_QV}\\
\int_X\left|F'(v;f(v))+\frac{1}{2}\sum_{|m|\geq 0}|a_m|^2F''(v;e_m,e_m)\right|\nu(dv) &\leq C.\label{condition_sur_Drift}
\end{align}
 Then for any non-negative function $g\in C_0^\infty(\R)$ we have  
\begin{equation}
\int_Og(F(v))\nu(dv)\leq \frac{C}{c}\int_{\R}g(x)dx.
\end{equation}
\end{thm}

\begin{proof}
Let $g$ be a positive $C_0^\infty$-function on $\R,$ set the function
\begin{align*}
\Phi_\lambda(x)=\frac{1}{\sqrt{2\lambda}}\int_\R g(y)e^{-|x-y|\sqrt{2\lambda}}dy=\frac{1}{\sqrt{2\lambda}}\left(\int_{-\infty}^xg(y)e^{-(x-y)\sqrt{2\lambda}}dy+\int_x^{\infty}g(y)e^{(x-y)\sqrt{2\lambda}}dy\right).
\end{align*}
Thanks to the properties of $g,$ we can differentiate this function and obtain
\begin{align*}
\Phi_\lambda'(x)=\int_x^{\infty}g(y)e^{(x-y)\sqrt{2\lambda}}dy-\int_{-\infty}^xg(y)e^{-(x-y)\sqrt{2\lambda}}dy.
\end{align*}
Computing the second derivative of $\Phi_\lambda$, we obtain that
\begin{align}
\frac{1}{2}\Phi_\lambda''+g=\lambda \Phi_\lambda.\label{equ_Phi_lamda}
\end{align}

We apply the Itô formula  to $\Phi_\lambda\circ F(u)$: 
\begin{align*}
\Phi_\lambda(F(u))&=\Phi_\lambda(F(u_0))\nonumber\\
&+\int_0^t\left(\Phi_\lambda'(F(u))\left\{F'(u,f(u))+\frac{1}{2}\sum_{|m|\geq 0}|a_m|^2F''(u;e_m,e_m)\right\}+\Phi_\lambda''(F(u))Q(u)\right)ds\nonumber\\ 
&+\sum_{|m|\geq 0}\int_0^t\Phi_\lambda'(F(u))F'(u,e_m)d\beta_m(s).
\end{align*}
Integrating with respect to $\nu$ and using its stationarity, we get
\begin{equation}\label{chap5_Balance_lambda}
\int_X\left(\Phi_\lambda'(F(v))\left\{F'(v,f(v))+\frac{1}{2}\sum_{|m|\geq 0}|a_m|^2F''(u;e_m,e_m)\right\}+\Phi_\lambda''(F(v))Q(v)\right)\nu(dv)=0.
\end{equation}
Now, evaluate the equation \eqref{equ_Phi_lamda} at the point $F(v),\ v\in O,$  multiply by $Q(v)$, then integrate over $O$ against $\nu$. Using \eqref{chap5_Balance_lambda}, we find
\begin{align*}
\int_OQ(v) &g(F(v))\nu(dv) =\int_O\lambda \Phi_\lambda(F(v))Q(v)\nu(dv)
-\frac{1}{2}\int_O\Phi_\lambda''(F(v))Q(v)\nu(dv)\\
&=\int_O\lambda \Phi_\lambda(F(v))Q(v)\nu(dv)
-\frac{1}{2}\int_X\Phi_\lambda''(F(v))Q(v)\nu(dv)\\
&+\frac{1}{2}\int_{X\backslash O}\Phi_\lambda''(F(v))Q(v)\nu(dv)\\
&=\int_O\lambda \Phi_\lambda(F(v))Q(v)\nu(dv)
+\frac{1}{2}\int_X\Phi_\lambda'(F(v))\left\{F'(v,f(v))+\frac{1}{2}\sum_{|m|\geq 0}|a_m|^2F''(v;e_m,e_m)\right\}\nu(dv)\\
&+\frac{1}{2}\int_{X\backslash O}\Phi_\lambda''(F(v))Q(v)\nu(dv).
\end{align*}
Now, in view of the definition of $\Phi_\lambda$, we see clearly that
$$\lambda\Phi_\lambda(x)\to 0\quad as\ \lambda\to 0,\quad \forall x\in\R.$$
Also, as $\lambda\to 0,$ we have, using the non-negativity of $g$, that
$$\Phi_\lambda'(x)\to \int_x^{\infty}g(y)dy -\int_{-\infty}^xg(y)dy\leq \int_x^{\infty}g(y)dy \leq \int_{\R}g(y)dy\quad \forall x\in\R,$$
and, using again the sign of $g$, we obtain as $\lambda\to 0$
$$\Phi_\lambda''(x)\to -2g(x)\leq 0\quad \forall x\in\R.$$
Finally, with the use of the Lebesgue's dominated convergence theorem, we arrive at
\begin{align*}
\int_O g(F(v))\nu(dv)\leq\frac{1}{2}\int_{X}\left|F'(v,f(v))+\frac{1}{2}\sum_{|m|\geq 0}|a_m|^2F''(v;e_m,e_m)\right|\nu(dv)\int_\R g(x)dx.
\end{align*}
It remains to use the hypothesis \eqref{condition_sur_QV} and \eqref{condition_sur_Drift} to obtain the claim.
\end{proof}
Let us consider now the NLS equation for which we have constructed an invariant measure $\mu$. Let $M(u)$ and $E(u)$ be the mass and energy of the equation.
\begin{cor}\label{cor_density_general}
Suppose the numbers $a_m$ are nonzero for all indices.
The laws under $\mu$ of the quantities $M(u)$ and $E(u)$ are absolutely continuous with respect to the Lebesgue measure on $\R\backslash (-a,a)$ for any $a>0$. More precisely, there is a positive constant $C(a)$ such that for any Borel set $\Gamma\subset\R\backslash (-a,a),$
\begin{equation}
M_*\mu(\Gamma),\ F_*\mu(\Gamma)\leq C(a)\ell(\Gamma).\label{est_density_final}
\end{equation}
\end{cor}
\begin{proof}
It suffices to prove \eqref{est_density_final} for the measures $\mu_{\al, N}$ where $C$ is independent of $\al$ and $N$. Indeed, once we have such a bound, we can finish the argument by invoking the Portmanteau theorem.\\ 
Since the measure $\mu_{\al,N}$ is concentrated on $E_N,$ let us set $X=E_N$
and $B_\epsilon$ be the closed ball in $E_N$, with center $0$ and radius $0<\epsilon<1.$\\
Set the quadratic variations for $M(u)$ and $E(u)$:
\begin{align*}
\frac{\al}{2}Q_M(u) &=\frac{\al}{2}\sum_{|m|\leq N}|a_m|^2(u,e_m)^2,\\
\frac{\al}{2}Q_E(u) &=\frac{\al}{2}\sum_{|m|\leq N}|a_m|^2(-\Delta u +|u|^{p-1}u,e_m)^2.
\end{align*}
Since $a_m\neq 0$ for any $m$, we see that $Q_M$ and $Q_E$ vanish only at $0.$\\ In what follows the symbol $Q$ denote both $Q_M$ and $Q_E$.
 We claim that \eqref{condition_sur_QV} holds on the set $O=X\backslash B_\epsilon$ for $Q$ with a constant $c=c(\epsilon).$ Indeed, since $Q(v)=0$ only for $v=0$ and $B_\epsilon$ is compact in $X$, we have, from the continuity of $Q$ on $B_\epsilon$, that $Q(B_\epsilon)$ is a compact interval (non reduced to $\{0\}$ because of the non-vanishing property of $Q$ outside $0$) $I_\epsilon\subset [0,\infty)$ containing $0$.  Therefore, if we denote by $c(\epsilon)>0$ the upper point of $I_\epsilon$, we have that
\begin{equation}
Q(v)\geq c(\epsilon) \quad \text{for any $v\in X\backslash B_\epsilon.$}
\end{equation}
Therefore 
\begin{equation}\label{trick_in_density_Qgeqceps}
\frac{\al}{2}Q(v)\geq \frac{\al}{2}c(\epsilon) \quad \text{for any $v\in X\backslash B_\epsilon.$}
\end{equation}
Now, using Theorem \ref{thm_density_general}, we claim that for a constant $C$ independent of $\al$ and $N,$ we have
\begin{align*}
\int_{X}g(F(v))\mu_{\al,N}(dv) &=\int_{X\backslash B_\epsilon}g(F(v))\mu_{\al,N}(dv)+ \int_{B_\epsilon}g(F(v))\mu_{\al,N}(dv)\\
&\leq \frac{C\al}{c(\epsilon)}\int_\R g(x)dx+\int_{B_\epsilon}g(F(v))\mu_{\al,N}(dv).
\end{align*}
Indeed, according to \eqref{condition_sur_Drift}, $C$ must be a bound for the following quantities (drifts of $M(u)$ and $E(u)$):
\begin{align*}
\E|M'(u,\i[(\Delta-1) u-P_N(|u|^{p-1}u)]-\al[(1-\Delta)^{s-1}+e^{\rho(\|u\|_{s-})}]u)|=\al\E\mathcal{M}(u),
\end{align*}
or (using \eqref{DefmathcalE})
\begin{align*}
\E|E'(u,\i[(\Delta-1) u-P_N(|u|^{p-1}u)]-\al[(1-\Delta)^{s-1}+e^{\rho(\|u\|_{s-})}]u)|\leq \al\E\mathcal{E}(u),
\end{align*}
depending on the functional we consider.
But in both cases, the estimates \eqref{est_StatMeas_mathcalM} and \eqref{est_StatMeas_mathcalE} provide bounds for $\E\mathcal{M}(u),\ \E\mathcal{E}(u)$ that are independent of both $\al$ and $N.$ Then we consider such bounds $C$.\\
By an standard approximation argument, we pass from $C_0^\infty$-functions to indicator functions in the above inequality. We arrive at, for every $a>0,$ for every Borel set $\Gamma$ in $\R$ contained in $\R\backslash (-a,a),$ and for any $\epsilon>0,$                                                                            
\begin{align*}
F_*\mu_{\al,N}(\Gamma)\leq \frac{C}{c(\epsilon)}\ell(\Gamma)+F_*\nu(\Gamma\cap [-\epsilon,\epsilon]).
\end{align*}
Choosing $\epsilon=a/2$, we obtain
\begin{align}
F_*\mu_{\al,N}(\Gamma)\leq C(a)\ell(\Gamma)\quad \text{for any Borel set }  \Gamma\subset\R\backslash(-a,a).\label{estimee_de_la_density_restriction}
\end{align}  
\end{proof}

Let us present here a result of estimation of the measure $\mu$ around $0.$  The strategy of its proof is due to Shirikyan \cite{armen_nondegcgl} and uses the properties of the local time of a functional based on the $L^2-$norm of the fluctuation-dissipation stationary solutions. The preservation of this norm by the limitting flow is  crucial to obtain uniform bounds that allow to pass to the limit, we refer to \cite{armen_nondegcgl} for a complete proof.
\begin{prop}\label{Shirikyan}
Let $\lambda_m\neq 0$, at least for two indices.
There is a constant $C>0$ such that
\begin{align}
\mu(B_\delta(L^2))\leq C\delta,\quad for\ any\ \delta>0.\label{Noatom}
\end{align}
\end{prop}

\begin{proof}[Proof of Theorem \ref{thm_density}]
For $\Gamma \in \text{Bor}(\R)$, let $\delta >0$, we write
\begin{align*}
M_*\mu(\Gamma)=M_*\mu(\Gamma\cap [-\delta,\delta])+M_*\mu(\Gamma\cap(\R\backslash [-\delta,\delta])).
\end{align*}
It remains to apply the Corollary \ref{cor_density_general}, and the Proposition \ref{Shirikyan} to obtain the claimed absolute continuity for $M_*\mu$. We do the same for the measure $E_*\mu.$ That finishes the proof.
\end{proof}

\section*{Acknowledgment.}
The author is indebted to Armen Shirikyan and Nikolay Tzvetkov for valuable discussions and remarks  that have been very beneficial for this work. He also thanks Gigliola Staffilani for pointing him out the very interesting problem of energy supercritical PDEs.  He is very grateful as well to an anonymous referee and to Bjoern Bringmann for pointing out a gap in a earlier version of the paper.

\bibliography{superKG}

\newcommand{\etalchar}[1]{$^{#1}$}
\begin{thebibliography}{BDSW19}

\bibitem[BB14a]{bourbulNLS}
J.~Bourgain and A.~Bulut.
\newblock Almost sure global well-posedness for the radial nonlinear
  {S}chr\"{o}dinger equation on the unit ball {II}: the 3d case.
\newblock {\em J. Eur. Math. Soc. (JEMS)}, 16(6):1289--1325, 2014.

\bibitem[BB14b]{bourbulW}
J.~Bourgain and A.~Bulut.
\newblock Invariant {G}ibbs measure evolution for the radial nonlinear wave
  equation on the 3d ball.
\newblock {\em J. Funct. Anal.}, 266(4):2319--2340, 2014.

\bibitem[BDSW19]{beceanu}
M.~Beceanu, Q.~Deng, A.~Soffer, and Y.~Wu.
\newblock Large global solutions for nonlinear {S}chrödinger equations iii,
  energy-supercritical cases.
\newblock {\em arXiv preprint arXiv:1901.07709}, 2019.

\bibitem[BGT04]{BGT04}
N.~Burq, P.~G{\'e}rard, and N.~Tzvetkov.
\newblock Strichartz inequalities and the nonlinear schr{\"o}dinger equation on
  compact manifolds.
\newblock {\em American Journal of Mathematics}, 126(3):569--605, 2004.

\bibitem[BGT05]{bgtGS}
N.~Burq, P.~G\'{e}rard, and N.~Tzvetkov.
\newblock Multilinear eigenfunction estimates and global existence for the
  three dimensional nonlinear {S}chr\"{o}dinger equations.
\newblock {\em Ann. Sci. \'{E}cole Norm. Sup. (4)}, 38(2):255--301, 2005.

\bibitem[BGT07]{bgtcoll}
N.~Burq, P.~G\'{e}rard, and N.~Tzvetkov.
\newblock Global solutions for the nonlinear {S}chr\"{o}dinger equation on
  three-dimensional compact manifolds.
\newblock In {\em Mathematical aspects of nonlinear dispersive equations},
  volume 163 of {\em Ann. of Math. Stud.}, pages 111--129. Princeton Univ.
  Press, Princeton, NJ, 2007.

\bibitem[Bou93]{bourfocS}
J.~Bourgain.
\newblock Fourier transform restriction phenomena for certain lattice subsets
  and applications to nonlinear evolution equations. {I}. {S}chr\"{o}dinger
  equations.
\newblock {\em Geom. Funct. Anal.}, 3(2):107--156, 1993.

\bibitem[Bou94]{bourg94}
J.~Bourgain.
\newblock Periodic nonlinear {S}chr\"{o}dinger equation and invariant measures.
\newblock {\em Comm. Math. Phys.}, 166(1):1--26, 1994.

\bibitem[Bou96]{bourgNLS96}
J.~Bourgain.
\newblock Invariant measures for the {$2$}{D}-defocusing nonlinear
  {S}chr\"{o}dinger equation.
\newblock {\em Comm. Math. Phys.}, 176(2):421--445, 1996.

\bibitem[Bou99a]{bourgain99}
J.~Bourgain.
\newblock {\em Global solutions of nonlinear Schr{\"o}dinger equations},
  volume~46.
\newblock American Mathematical Soc., 1999.

\bibitem[Bou99b]{bourgcrtNLS}
J.~Bourgain.
\newblock Global wellposedness of defocusing critical nonlinear
  {S}chr\"{o}dinger equation in the radial case.
\newblock {\em J. Amer. Math. Soc.}, 12(1):145--171, 1999.

\bibitem[BT07]{bt2007}
N.~Burq and N.~Tzvetkov.
\newblock Invariant measure for a three dimensional nonlinear wave equation.
\newblock {\em International Mathematics Research Notices},
  2007(9):rnm108--rnm108, 2007.

\bibitem[BT08a]{btrandom1}
N.~Burq and N.~Tzvetkov.
\newblock Random data {C}auchy theory for supercritical wave equations {I}:
  local theory.
\newblock {\em Inventiones mathematicae}, 173(3):449--475, 2008.

\bibitem[BT08b]{btrandom2}
N.~Burq and N.~Tzvetkov.
\newblock Random data {C}auchy theory for supercritical wave equations. {II}.
  {A} global existence result.
\newblock {\em Invent. Math.}, 173(3):477--496, 2008.

\bibitem[BT14]{burktzvt_probwav}
N.~Burq and N.~Tzvetkov.
\newblock Probabilistic well-posedness for the cubic wave equation.
\newblock {\em J. Eur. Math. Soc. (JEMS)}, 16(1):1--30, 2014.

\bibitem[BTT18]{btt18}
N.~Burq, L.~Thomann, and N.~Tzvetkov.
\newblock Remarks on the {G}ibbs measures for nonlinear dispersive equations.
\newblock {\em Ann. Fac. Sci. Toulouse Math. (6)}, 27(3):527--597, 2018.

\bibitem[Caz03]{caz}
T.~Cazenave.
\newblock {\em Semilinear {S}chr\"{o}dinger equations}, volume~10 of {\em
  Courant Lecture Notes in Mathematics}.
\newblock New York University, Courant Institute of Mathematical Sciences, New
  York; American Mathematical Society, Providence, RI, 2003.

\bibitem[CKS{\etalchar{+}}08]{CKSTTcrtNLS}
J.~Colliander, M.~Keel, G.~Staffilani, H.~Takaoka, and T.~Tao.
\newblock Global well-posedness and scattering for the energy-critical
  nonlinear {S}chr\"{o}dinger equation in {$\Bbb R^3$}.
\newblock {\em Ann. of Math. (2)}, 167(3):767--865, 2008.

\bibitem[CO12]{colloh}
J.~Colliander and T.~Oh.
\newblock Almost sure well-posedness of the cubic nonlinear {S}chr\"odinger
  equation below {$L^2(\Bbb T)$}.
\newblock {\em Duke Math. J.}, 161(3):367--414, 2012.

\bibitem[CW90]{cazweis}
T.~Cazenave and F.~B. Weissler.
\newblock The cauchy problem for the critical nonlinear schr{\"o}dinger
  equation in hs.
\newblock {\em Nonlinear Analysis: Theory, Methods \& Applications},
  14(10):807--836, 1990.

\bibitem[DPZ14]{da2014stochastic}
G.~Da~Prato and J.~Zabczyk.
\newblock {\em Stochastic equations in infinite dimensions}.
\newblock Cambridge {U}niversity {P}ress, Cambridge, 2014.

\bibitem[dS11]{asds}
A.-S. de~Suzzoni.
\newblock Invariant measure for the cubic wave equation on the unit ball of
  {$\Bbb R^3$}.
\newblock {\em Dyn. Partial Differ. Equ.}, 8(2):127--147, 2011.

\bibitem[Dud02]{dudley}
R.~M. Dudley.
\newblock {\em Real analysis and probability}.
\newblock Cambridge University Press, 2002.

\bibitem[Gri00]{GrllakscrtNLS}
M.~G. Grillakis.
\newblock On nonlinear {S}chr\"{o}dinger equations.
\newblock {\em Comm. Partial Differential Equations}, 25(9-10):1827--1844,
  2000.

\bibitem[GV79]{gv}
J.~Ginibre and G.~Velo.
\newblock On a class of nonlinear {S}chr\"{o}dinger equations. {I}. {T}he
  {C}auchy problem, general case.
\newblock {\em J. Funct. Anal.}, 32(1):1--32, 1979.

\bibitem[HTT11]{htt}
S.~Herr, D.~Tataru, and N.~Tzvetkov.
\newblock Global well-posedness of the energy-critical nonlinear
  {S}chr\"{o}dinger equation with small initial data in {$H^1(\Bbb T^3)$}.
\newblock {\em Duke Math. J.}, 159(2):329--349, 2011.

\bibitem[IP12]{ionpaus}
A.~D. Ionescu and B.~Pausader.
\newblock The energy-critical defocusing {NLS} on {$\Bbb T^3$}.
\newblock {\em Duke Math. J.}, 161(8):1581--1612, 2012.

\bibitem[KM06]{kenigmerle}
C.~E. Kenig and F.~Merle.
\newblock Global well-posedness, scattering and blow-up for the
  energy-critical, focusing, non-linear {S}chr\"{o}dinger equation in the
  radial case.
\newblock {\em Invent. Math.}, 166(3):645--675, 2006.

\bibitem[Kry95]{krylov95}
N.~V. Krylov.
\newblock {\em Introduction to the theory of diffusion processes}.
\newblock Amer Mathematical Society, 1995.

\bibitem[KS04]{KS04}
S.~Kuksin and A.~Shirikyan.
\newblock Randomly forced {CGL} equation: stationary measures and the inviscid
  limit.
\newblock {\em Journal of Physics}, 37:3805--3822, 2004.

\bibitem[KS12]{KS12}
S.~Kuksin and A.~Shirikyan.
\newblock {\em Mathematics of Two-Dimensional Turbulence}.
\newblock Cambridge University Press, Cambridge, 2012.

\bibitem[Kuk04]{kuk_eul_lim}
S.~Kuksin.
\newblock The {E}ulerian limit for 2{D} statistical hydrodynamics.
\newblock {\em J. Statist. Phys.}, 115(1-2):469--492, 2004.

\bibitem[Kuk08]{kuksin_nondegeul}
{S.} Kuksin.
\newblock On distribution of energy and vorticity for solutions of 2d
  {N}avier-{S}tokes equation with small viscosity.
\newblock {\em Communications in Mathematical Physics}, 284(2):407--424, 2008.

\bibitem[LM72]{LM}
J.~L. Lions and E.~Magenes.
\newblock {\em Non-Homogeneous Boundary Value Problems and Applications},
  volume~1.
\newblock Springer-Verlag, 1972.

\bibitem[LRS88]{LRS}
J.~L. Lebowitz, H.~A. Rose, and E.~R. Speer.
\newblock Statistical mechanics of the nonlinear {S}chr\"{o}dinger equation.
\newblock {\em J. Statist. Phys.}, 50(3-4):657--687, 1988.

\bibitem[NPS13]{MR3131480}
A.~R. Nahmod, N.~Pavlovi\'{c}, and G.~Staffilani.
\newblock Almost sure existence of global weak solutions for supercritical
  {N}avier-{S}tokes equations.
\newblock {\em SIAM J. Math. Anal.}, 45(6):3431--3452, 2013.

\bibitem[OP16]{ohpocov}
T.~Oh and O.~Pocovnicu.
\newblock Probabilistic global well-posedness of the energy-critical defocusing
  quintic nonlinear wave equation on r3.
\newblock {\em Journal de Math{\'e}matiques Pures et Appliqu{\'e}es},
  105(3):342--366, 2016.

\bibitem[OST18]{tzQI3}
T.~Oh, P.~Sosoe, and N.~Tzvetkov.
\newblock An optimal regularity result on the quasi-invariant {G}aussian
  measures for the cubic fourth order nonlinear {S}chr\"{o}dinger equation.
\newblock {\em J. \'{E}c. polytech. Math.}, 5:793--841, 2018.

\bibitem[OT17]{tzQI2}
T.~Oh and N.~Tzvetkov.
\newblock Quasi-invariant {G}aussian measures for the cubic fourth order
  nonlinear {S}chr\"{o}dinger equation.
\newblock {\em Probab. Theory Related Fields}, 169(3-4):1121--1168, 2017.

\bibitem[Poc14]{pocovnicu2014}
O.~Pocovnicu.
\newblock Almost sure global well-posedness for the energy-critical defocusing
  nonlinear wave equation on {R}d, d= 4 and 5.
\newblock {\em arXiv preprint arXiv:1406.1782}, 2014.

\bibitem[PTW14]{ptw}
B.~Pausader, N.~Tzvetkov, and X.~Wang.
\newblock Global regularity for the energy-critical {NLS} on {$\Bbb{S}^3$}.
\newblock {\em Ann. Inst. H. Poincar\'{e} Anal. Non Lin\'{e}aire},
  31(2):315--338, 2014.

\bibitem[Shi11]{armen_nondegcgl}
A.~Shirikyan.
\newblock Local times for solutions of the complex {G}inzburg-{L}andau equation
  and the inviscid limit.
\newblock {\em J. Math. Anal. Appl.}, 384(1):130--137, 2011.

\bibitem[Sy18]{sybo}
M.~Sy.
\newblock Invariant measure and long time behavior of regular solutions of the
  {B}enjamin--{O}no equation.
\newblock {\em Analysis \& PDE}, 11(8):1841--1879, 2018.

\bibitem[Sy19]{sykg}
M.~Sy.
\newblock Invariant measure and large time dynamics of the cubic klein--gordon
  equation in 3d.
\newblock {\em Stochastics and Partial Differential Equations: Analysis and
  Computations}, 7(3):379--416, 2019.

\bibitem[Tho09]{thomann2009random}
L.~Thomann.
\newblock Random data {C}auchy problem for supercritical {S}chr{\"o}dinger
  equations.
\newblock In {\em Annales de l'Institut Henri Poincare (C) Non Linear
  Analysis}, volume~26, pages 2385--2402. Elsevier, 2009.

\bibitem[Tzv06]{tzvNLS06}
N.~Tzvetkov.
\newblock Invariant measures for the nonlinear {S}chr\"{o}dinger equation on
  the disc.
\newblock {\em Dyn. Partial Differ. Equ.}, 3(2):111--160, 2006.

\bibitem[Tzv08]{tzvNLS}
N.~Tzvetkov.
\newblock Invariant measures for the defocusing nonlinear {S}chr\"{o}dinger
  equation.
\newblock {\em Ann. Inst. Fourier (Grenoble)}, 58(7):2543--2604, 2008.

\bibitem[Tzv15]{tzvQI1}
N.~Tzvetkov.
\newblock Quasiinvariant {G}aussian measures for one-dimensional {H}amiltonian
  partial differential equations.
\newblock {\em Forum Math. Sigma}, 3:e28, 35, 2015.

\bibitem[Wan16]{wang}
W.-M. Wang.
\newblock Energy supercritical nonlinear {S}chrödinger equations:
  Quasiperiodic solutions.
\newblock {\em Duke Math. J.}, 165(6):1129--1192, 04 2016.

\bibitem[Zhi91]{zhidnls}
P.~E. Zhidkov.
\newblock An invariant measure for the nonlinear {S}chr\"{o}dinger equation.
\newblock {\em Dokl. Akad. Nauk SSSR}, 317(3):543--546, 1991.

\bibitem[Zhi94]{zhidwave}
P.~E. Zhidkov.
\newblock An invariant measure for a nonlinear wave equation.
\newblock {\em Nonlinear Anal.}, 22(3):319--325, 1994.

\bibitem[Zhi01]{zhd}
P.~E. Zhidkov.
\newblock {\em Korteweg-de {V}ries and nonlinear {S}chr\"{o}dinger equations:
  qualitative theory}, volume 1756 of {\em Lecture Notes in Mathematics}.
\newblock Springer-Verlag, Berlin, 2001.

\end{thebibliography}
\bibliographystyle{alpha}
\end{document}